\documentclass[11pt,reqno]{amsart}
\usepackage{amsfonts,amssymb,amsmath,amsthm,adjustbox}
\usepackage[margin=1in]{geometry}
\usepackage[hidelinks]{hyperref}
\usepackage{enumerate}
\usepackage{microtype}
\usepackage{mathtools}
\usepackage{cite}
\usepackage{tikz-cd}
\usepackage{float}
\usepackage{cleveref}

\theoremstyle{plain}
\newtheorem{theorem}{Theorem}
\newtheorem{corollary}[theorem]{Corollary}
\newtheorem{lemma}[theorem]{Lemma}
\newtheorem{proposition}[theorem]{Proposition}
\theoremstyle{definition}
\newtheorem{remark}[theorem]{Remark}

\newtheorem{example}[theorem]{Example}

\newtheorem{question}[theorem]{Question}

\def\Q{\mathbb{Q}}
\def\Z{\mathbb{Z}}
\def\F{\mathbb{F}}
\def\R{\mathbb{R}}

\def\Gal{\operatorname{Gal}}
\newcommand{\neghs}{\mkern-12mu}

\DeclareMathOperator{\GL}{GL}
\DeclareMathOperator{\Aut}{Aut}
\DeclareMathOperator{\Tr}{tr }
\DeclareMathOperator{\adj}{adj}
\DeclareMathOperator{\divi}{\!-div}
\DeclareMathOperator{\Frob}{Frob}

\theoremstyle{remark}

\title{Divisibility Biases in the Orders of Elliptic Curve Reductions}
\date{\today}


\author{Sung Min Lee}
\address{Sung Min Lee, Department of Mathematics, Wake Forest University, Winston-Salem, NC 27109}

\author{Nara Sheen}
\address{Nara Sheen, Department of Mathematics and Statistics, Minnesota State University, Mankato, Mankato, MN 56001}

\begin{document}

\begin{abstract}
Let $E$ be an elliptic curve over the rationals. In 2004, Cojocaru proved, using the Chebotarev density theorem, that the set of primes $p \leq x$ for which $m$ divides $\#E_p(\F_p)$ has a natural density. In 2009, Banks and Shparlinski proved an averaged version of this result over families of elliptic curves. In this article, we give a more explicit analysis of these densities. In particular, we show that, for Serre curves, the density of primes $p$ for which $m \mid \#E_p(\F_p)$ is approximately $1/\varphi(m)$, and is always greater than $1/m$ for every $m \geq 2$. Thus, the orders $\#E_p(\F_p)$ exhibit a bias toward divisibility by $m$. Finally, based on Jones' method, we prove that the average of the individual $m$-divisibility densities coincides with the average density proposed by Banks and Shparlinski.
\end{abstract}

\maketitle

\section{Introduction}
Let $E/\Q$ be an elliptic curve of conductor $N_E$. The curve $E$ admits a short Weierstrass model; that is, there exist $a,b \in \Z$ for which
\begin{equation}\label{shortWeierstrass}
    E : Y^2 = X^3 + aX+b,
\end{equation}
where $\Delta_E \coloneqq -16(4a^3+27b^2) \neq 0$. For a rational prime $p \nmid N_E$, which we call a \textit{good prime}, let $E_p$ denote the reduction of $E$ modulo $p$:
$$E_p : Y^2 \equiv X^3 + aX +b \pmod p.$$
The set of $\F_p$-rational points on $E_p$, denoted by $E_p(\F_p)$, forms a finite abelian group. The Frobenius trace $a_p(E)$ is defined by
\begin{equation}\label{Frobtrace}
    a_p(E)=p+1-\#E_p(\F_p).
\end{equation}
The Hasse--Weil bound states that $|a_p(E)| \leq 2\sqrt{p}$, implying that $p$ and $\#E_p(\F_p)$ are of comparable size.

For a fixed $E/\Q$, the distribution of the group structures of $E_p(\F_p)$ as $p$ varies has been long studied by numerous mathematicians. One of the earliest questions, first studied by Serre \cite[pp. 465–466]{MR3223094}, asks how often $E_p(\F_p)$ is cyclic. This problem is closely related to the elliptic curve analogue of Artin's primitive root conjecture which we discuss in Remark \ref{naturalperspective}. Another important example is the Koblitz--Zywina conjecture \cite{MR917870,MR2805578}, which asks how often $\#E_p(\F_p)$ is prime; this is often viewed as an elliptic curve analogue of the twin prime conjecture and the first Hardy--Littlewood conjecture. This connection is not limited to a single elliptic curve case. For example, given two elliptic curves $E$ and $E'$, one may ask how often $\#E_p(\F_p)$ and $\#E_p'(\F_p)$ are coprime. This problem can be viewed as an elliptic curve analogue of the classical coprimality problem for integers and was recently studied in \cite{hamakiotes2026densitycoprimereductionselliptic}.

In this article, we study how the group order $\#E_p(\F_p)$ is distributed as $p$ varies. More specifically, for a fixed elliptic curve $E/\Q$ and a positive integer $m$, we ask how often $\#E_p(\F_p)$ is divisible by $m$. We call such primes $p$ \textit{$m$-divisibility primes} for $E$. This question was first studied by Cojocaru.
\begin{theorem}[Cojocaru, \protect{\cite[Theorem 1]{MR2076566}}]\label{cojocaru}
    Let $E/\Q$ be an elliptic curve of conductor $N_E$, and let $m$ be a positive integer. Then there exists a constant $C_E^{m\divi} \geq 0$ such that
    $$\pi_E^{m\divi}(x) \coloneqq \#\{p \leq x : p\nmid N_E, m \mid \#E_p(\F_p)\} \sim C^{m\divi}_E \cdot \frac{x}{\log x}, \quad \text{ as } x\to \infty.$$
\end{theorem}
The density is given in \eqref{CmdivE}. For a fixed elliptic curve $E/\Q$, the condition $m \mid \#E_p(\F_p)$ can be interpreted as a condition on the Frobenius element $\Frob_p$ in a suitable Galois extension of $\Q$. More precisely, it corresponds to a subset closed under conjugation of the relevant Galois group. Thus, Theorem \ref{cojocaru} is a standard application of the Chebotarev density theorem. We discuss this interpretation in more detail in Section \ref{mdivisibilityprime}.

A naive heuristic suggests that the integers $\#E_p(\F_p)$ should behave like random integers of comparable size. From this point of view, one would expect the density $C^{m\divi}_E$ to be approximately $1/m$, or at least to have the same order of magnitude. Cojocaru also confirmed that the density should be approximately $1/m$ in a generic case.

There is also a ``local'' analogue of this question. Fix a prime power $q$, and let $E$ vary over elliptic curves over $\F_q$. Since the condition $m \mid \#E(\F_q)$ is preserved under $\F_q$-isomorphism, we count each $\F_q$-isomorphism class $[E]$ with weight $1/\#\Aut_{\F_q}(E)$. One may then ask for the weighted probability that $\#E(\F_q)$ is divisible by a fixed positive integer $m$.

Lenstra Jr. \cite[Proposition 1.14]{MR916721}first studied this question in the case where $q>3$ is prime and $m=\ell$ is a prime distinct from $q$. Howe \cite{MR1204781} later generalized this result to arbitrary positive integers $m$ and prime powers $q$. Howe explicitly computed the weighted probability $\omega_q(m)$ with a concrete error bound, and observed that $\omega_q(m)$ is close to $1/\varphi(m)$. This suggests that the $m$-divisibility condition is biased above the naive random-integer prediction $1/m$. We discuss this more carefully in Remark \ref{HoweWork}.

There is an average result in this direction. Instead of fixing a single elliptic curve over $\Q$, one may average over a box of Weierstrass equations. Fix positive real numbers $A$ and $B$, and let $\mathcal{F}\coloneqq\mathcal{F}(A,B)$ denote the family of elliptic curves of the form \eqref{shortWeierstrass} with $|a| \leq A$ and $|b| \leq B$. Based on Howe's results, Banks and Shparlinski obtained the following global statistical result.
\begin{theorem}[Banks--Shparlinski, \protect{\cite[Theorem 19]{MR2570668}}]
    Let $\epsilon > 0$ and $K > 0$. Let $A$ and $B$ satisfy
    $$x^\epsilon \leq A,B \leq x^{1-\epsilon} \quad \text{ and } \quad x^{1+\epsilon} \leq AB.$$
    Then for any $m \leq \log^Kx$, we have
    $$\frac{1}{|\mathcal{F}|}\sum_{E\in \mathcal{F}} \pi_E^{m\divi}(x) \sim C^{m\divi} \cdot \frac{x}{\log x}, \quad \text{ as } x\to \infty,$$
    where $C^{m\divi}$ is defined in \eqref{definitionofaveragedensity}.
\end{theorem}
Banks and Shparlinski gave an explicit formula for the average density $C^{m\divi}$. However, they did not simplify this formula, and hence it is not immediately clear whether their result is closer to the random-integer prediction $1/m$ or to Howe's local weighted model $1/\varphi(m)$.

We now describe the main contributions of this article. Our first result is an explicit computation of $C^{m\divi}$, which shows that this average density is close to $1/\varphi(m)$ and thus supports Howe's weighted model. Building on work of Jones \cite{MR2534114}, we  derive an explicit formula for the $m$-divisibility density $C_E^{m\divi}$ for Serre curves $E$.

\begin{theorem}\label{maintheorem}
    Let $E/\Q$ be a Serre curve and let $m \geq 2$ be a positive integer. Let $\Delta'_E$ denote the squarefree part of the discriminant of $E$, and let $m_E$ denote the adelic level of $E$. Let $v_2(\cdot)$ denote the $2$-adic valuation. Suppose one of the following cases holds:
    \begin{itemize}
        \item $m_E \nmid m$.
        \item $m_E \mid m$, $\Delta'_E \equiv 3 \pmod 4$, and $v_2(m) = 2$.
        \item $m_E \mid m, \Delta'_E \equiv 2 \pmod 4$, and $v_2(m) \in \{3,4\}$.
    \end{itemize}
    Then we have
    $$C_E^{m\divi} = C^{m\divi}.$$
    Otherwise, we write $m = m_1m_2$ such that $m_1 = \gcd(m,m_E^\infty)$ and $\gcd(m_2,m_E) = 1$. Then
    $$C_E^{m\divi} = \left(C^{m_1\divi} + \prod_{\ell^\alpha \parallel m_1} \frac{-\ell^{2\alpha-1}}{|\GL_2(\Z/\ell^\alpha\Z)|}\right) \cdot C^{m_2\divi}.$$
\end{theorem}
A direct calculation from Theorem \ref{maintheorem} gives the following consequence, in line with Howe's local weighted model.
\begin{corollary}\label{maincorollary}
    Let $E/\Q$ be a Serre curve and $m \geq 2$ be any positive integer. Then $C_E^{m\divi} > 1/m$.
\end{corollary}

Theorem \ref{maintheorem} shows that, for Serre curves, the $m$-divisibility density $C_E^{m\divi}$ differs from the average density $C^{m\divi}$ by an explicit, typically small correction term that depends on $m$ and $m_E$. Since Jones \cite{MR2563740} proved that almost all elliptic curves are Serre curves, it is natural to expect that these correction terms vanish on average over large boxes $\mathcal{F} \coloneqq \mathcal{F}(A,B)$ of Weierstrass equations. Our second main result confirms this. We first recall the $k$-th power divisor sum:
$$\sigma_k(m) \coloneqq \sum_{d\mid m} d^k.$$

\begin{theorem}\label{maintheorem2}
    Fix positive integers $m \geq 2$ and $k \geq 1$. Then there exists $\gamma > 0$ independent of $m$ and $k$, such that
    $$\frac{1}{|\mathcal{F}|}\sum_{E \in \mathcal{F}}|C_E^{m\divi}-C^{m\divi}|^k \ll_k \frac{\sigma_{-k}(m)}{A\cdot \varphi(m)^k} + \frac{\log B \cdot \log^7A \cdot \sigma_{1-k}(m)}{B \cdot \varphi(m)^k}  + \frac{\log^\gamma(\min \{A,B\})}{\sqrt{\min\{A,B\}}}.$$
    In particular, if $m$ is odd or $m \in \{2,4,8,16\}$, then
    $$\frac{1}{|\mathcal{F}|}\sum_{E \in \mathcal{F}}|C_E^{m\divi}-C^{m\divi}|^k \ll \frac{\log^\gamma(\min \{A,B\})}{\sqrt{\min\{A,B\}}}.$$
\end{theorem}

\begin{remark}
    Let $m$ and $k$ be chosen as in Theorem \ref{maintheorem2}. Observe that
    $$\frac{\sigma_{1-k}(m)}{\varphi(m)^k} \leq \frac{\sigma_0(m)}{\varphi(m)} = \prod_{\ell^\alpha \parallel m} \frac{\alpha+1}{\ell^{\alpha-1}(\ell-1)} \leq 2,$$
    since 
    $$\frac{\sigma_0(\ell^\alpha)}{\varphi(\ell^\alpha)} \leq \begin{cases}
        2 & \text{ if } \ell = 2 \text{ and } \alpha \leq 2, \\
        1 & \text{ otherwise}. \end{cases}$$
    Theorem \ref{maintheorem2} states that if $A \coloneqq A(x)$ and $B \coloneqq B(x)$ both tend to infinity with $x$, then provided that 
    $$\lim_{x\to \infty} \frac{\log B \cdot \log^7 A}{B} = 0,$$
    we have
    $$\frac{1}{|\mathcal{F}|} \sum_{E\in \mathcal{F}} C_E^{m\divi} \to C^{m\divi},$$
    as $x \to \infty$. In particular, the average of the densities $C_E^{m\divi}$ agrees with the average density $C^{m\divi}$.
\end{remark}

\begin{remark}\label{naturalperspective}
The heuristic that $C_E^{m\divi}$ should be compared with $1/\varphi(m)$, rather than with $1/m$, is also natural from the classical analogy between $E_p(\F_p)$ and $\F_p^\times$. As will be discussed in more detail in Section \ref{mdivisibilityprime}, for a prime $p\nmid mN_E$, whether $p$ is an $m$-divisibility prime for $E$ is determined by the Frobenius element above $p$ in $\Q(E[m])$, the $m$-torsion field. More precisely, if $\rho_{E,m}$ denotes the mod $m$ Galois representation attached to $E$, then
\begin{equation}\label{ECsetting}
    \det(I-\rho_{E,m}(\Frob_p)) \equiv 0 \pmod m \iff m \mid \#E_p(\F_p)
\end{equation}
The classical number theoretic analogue is obtained by replacing the two-dimensional representation $\rho_{E,m}$ on $E[m]$ with the one-dimensional cyclotomic character $\chi_m$ on $\mu_m$, the $m$-th roots of unity. Then, \eqref{ECsetting} is replaced by 
$$  1 - \chi_m(\Frob_p) \equiv 0 \pmod m \iff m \mid \#\F_p^\times \iff p \equiv 1 \pmod m.$$
Thus, in the classical setting, the analogue of $\#E_p(\F_p)$ is not a randomly chosen integer of size roughly $p$, but rather the order of the multiplicative group $\#\F_p^\times=p-1$. The density of primes satisfying $p \equiv 1 \pmod m$ is $1/\varphi(m)$, which is close to the $m$-divisibility density in the generic elliptic curve case.

This analogy is not limited to the $m$-divisibility question. Artin's primitive root conjecture asks for the asymptotic behavior of
$$\#\{p \leq x : \langle a \pmod p\rangle = \F_p^\times\},$$
where $a$ is an integer which is neither $-1$ nor a perfect square. Lang and Trotter \cite{MR427273} formulated the corresponding elliptic curve analogue: given an elliptic curve $E/\Q$ and a point $P \in E(\Q)$ of infinite order, what is the asymptotic behavior of
$$\#\{p \leq x : \langle P \pmod p \rangle = E_p(\F_p)\}?$$
In this sense, the comparison between $E_p(\F_p)$ and $\F_p^\times$ has long served as a guiding principle in the study of reductions of elliptic curves.
\end{remark}

\begin{remark}\label{JonesRouse}
    Let $E/\Q$ be an elliptic curve of conductor $N_E$, and let $P\in E(\Q)$ satisfy $P \not \in 2E(\Q)$ and $P \not \in 2E(\Q(E[4])$. Jones and Rouse \cite[Theorem 1.1]{MR2640290} proved that the density of primes $p$ for which $P\bmod p$ has odd order in $E_p(\F_p)$ is $11/21$. Thus, the reductions of $P$ exhibit a bias toward having odd order. 

    In this remark, we show that our results provide a simple heuristic explanation for this phenomenon. Let $v_2$ denote the $2$-adic valuation. Let $G$ be a finite abelian group such that $v_2(|G|) = k$. Using an elementary algebraic argument, one can confirm that the proportion of elements of even order in $G$ is $1-1/2^k$.

    Suppose that $E$ has a surjective $2$-adic Galois representation. The density of primes $p$ for which $v_2(|E_p(\F_p)|) = k$ is given by 
    $$C_E^{2^{k}\divi} - C_E^{2^{k+1}\divi} = \frac{1}{\varphi(2^{k+1})}\left(1-\frac{1}{3\varphi(2^{k+1})}\right) - \frac{1}{\varphi(2^k)}\left(1-\frac{1}{3\varphi(2^k)}\right) = \frac{1}{2^k} - \frac{1}{4^k}.$$
    If one regards the reduction of $P$ modulo $p$ as a uniformly distributed element of $E_p(\F_p)$, then the expected density of primes for which $P\bmod p$ has even order is
    $$\sum_{k\geq 1} \text{Pr}(v_2(|E_p(\F_p)|=k) \cdot \text{Pr}(v_2(\text{ord}(x)\geq 1))= \sum_{k \geq 1} \left(\frac{1}{2^k}-\frac{1}{4^k}\right)\left(1-\frac{1}{2^k}\right) = \frac{10}{21}.$$
    The corresponding expected density for odd order is therefore $11/21$, agreeing with the theorem of Jones and Rouse.

In their paper, Jones and Rouse determined the density of primes $p$ for which the order of $P\bmod p$ is coprime to a fixed prime $\ell$, assuming that $\ell$-adic Galois representation is surjective and $P \not \in \ell E(\Q)$. The same elementary heuristic, applied to the $\ell$-primary part of $E_p(\F_p)$, reproduces their density for every prime $\ell$.

More broadly, understanding the distribution of $P\bmod p$ for a fixed point $P\in E(\Q)$ is a longstanding problem in arithmetic geometry, dating back to the primitive root conjecture of Lang and Trotter. While such questions are naturally governed by arboreal Galois representations, the calculation above suggests that useful intuition may be obtained by combining information about the distribution of the groups $E_p(\F_p)$ with the heuristic assumption that $P\bmod p$ behaves like a uniformly distributed element of $E_p(\F_p)$.
\end{remark}

The perspective in Remark \ref{naturalperspective} suggests further questions. Although $\#\F_p^\times=p-1$ is a natural analogue of $\#E_p(\F_p)$, there is an important difference between the two quantities modulo $m$. For $p \nmid m$, $p$ must be invertible modulo $m$, so $p-1$ cannot range freely over all residue classes modulo $m$. In contrast, $\#E_p(\F_p)$ is not subject to the same restriction. Thus, if the congruence class $0 \pmod m$ occurs more often than $1/m$, then some other congruence classes must occur less often. This leads to the following natural question.

\begin{question}\label{modmdist}
Fix an integer $r$ and a positive integer $m$. What is the asymptotic behavior of
$$\#\{p \leq x : p \nmid mN_E, \#E_p(\F_p) \equiv r \pmod m\}?$$
In particular, does this distribution exhibit a bias among congruence classes modulo $m$? Furthermore, is there a universal pattern?
\end{question}

There is also a closely related and deeper question. Reducing \eqref{Frobtrace} modulo $m$, we have
$$a_p(E) \equiv p+1 - \#E_p(\F_p) \pmod m.$$
By the prime number theorem for arithmetic progressions, the primes $p\nmid m$ are equidistributed among the invertible congruence classes modulo $m$. However, as observed from Theorem \ref{maintheorem} and Theorem \ref{maintheorem2}, the quantity $\#E_p(\F_p)$ is not expected to be uniformly distributed modulo $m$. Therefore, answering Question \ref{modmdist} may also shed light on the following question.

\begin{question}
Fix an integer $r$ and a positive integer $m$. What is the asymptotic behavior of
$$\#\{p \leq x : p \nmid mN_E, a_p(E) \equiv r\pmod m\} ?$$
In particular, does this distribution exhibit a bias among congruence classes modulo $m$? Furthermore, is there a universal pattern?
\end{question}

The distribution of Frobenius traces, for a fixed elliptic curve or for families of elliptic curves, has been studied extensively since the work of Lang and Trotter \cite{MR568299}. See \cite{MR1144318, MR1062334, MR1694997, MR3848226, MR935007, MR644559, MR3453123} for results concerning a single elliptic curve $E$, and \cite{MR1677267, MR3102527, MR2106483, MR2163519, MR3395787, MR1382477} for results concerning certain families of elliptic curves. By contrast, the distribution of $a_p(E)$ modulo a fixed integer $m$ appears to have received less systematic attention.

Nevertheless, some related congruence class phenomena have been studied. For instance, Jarov, Khadra, and Walji \cite{MR4810094} studied the distribution, in congruence classes of primes, of primes having a fixed Frobenius trace on average over thin families of elliptic curves. On the other hand, Jones and Vissuet \cite{MR4928015} explicitly studied elliptic curves whose Frobenius traces avoid certain congruence classes modulo $m$. The results of the present paper, together with an answer to Question \ref{modmdist}, may provide a starting point for a more systematic study of Frobenius traces modulo $m$.

The paper is organized as follows. In Section \ref{preliminaries}, we review the necessary background; in particular, we recall Galois representations, Serre curves, and known $m$-divisibility densities. We also prove Theorem \ref{maintheorem} in the case that $G_E(m)$ is the full group. In Section \ref{S:index2density}, we establish several lemmas and prove Theorem \ref{maintheorem} in the case that $E$ is a Serre curve and $G_E(m)$ is the index $2$ subgroup of $\GL_2(\Z/m\Z)$. In Section \ref{maintheorem2proof}, we prove Theorem \ref{maintheorem2} and give further analysis of the density $C_E^{m\divi}$.

Here is the list of notation that will be used throughout the paper:
\begin{itemize}
\itemindent=-13pt
\item We often write $(a,b)$ under a summation sign as shorthand for $\gcd(a,b)$.
\item We often write $a\equiv b \,(m)$ under a summation sign as shorthand for $a \equiv b \pmod m$.
\item For a commutative ring $R$, we write $M_2(R)$ for the additive group of all $2\times 2$ matrices with entries in $R$.
\item For an odd prime $\ell$, $\left(\frac{\cdot}{\ell}\right)$ denotes the Legendre symbol.
\item For an arbitrary group $G \subseteq \GL_2(\widehat{\Z})$, $G(m)$ denotes the image of $G$ under the reduction modulo $m$ map.
\item For functions $f,g : \R \to \R$, we write $f \ll g$ if there exist constants $C > 0$ and $x_0 \geq 0$ such that $|f(x)| \leq C|g(x)|$ for any $x > x_0$. If the constant $C$ depends on a parameter $k$, we write $f \ll_k g$.
\item For positive real numbers $A$ and $B$, $\mathcal{F}\coloneqq \mathcal{F}(A,B)$ denotes the family of models $Y^2=X^3+aX+b$ of elliptic curves for which $|a|\leq A$ and $|b| \leq B$.
\item Given a subfamily $\mathcal{G} \subseteq \mathcal{F}$ and functions $f,g :\mathcal{G} \to \R$, we write $f \ll g$ if there exists a constant $C > 0$ such that $|f(E)| \leq C|g(E)|$ for any $E \in \mathcal{G}$. If $C$ depends on a parameter $k$, we write $f \ll_k g$.
\item For a positive integer $k$, $\sigma_k$ denotes the $k$-th power divisor sum.
\item $\varphi$ denotes the Euler totient function.
\end{itemize}

\subsection*{Acknowledgements}
This project originated from the Informal Number Theory Series, an online seminar organized by the YouTube channel ``Enjoying Math''. The authors would like to thank Geonhee Cho and Garen Chiloyan for their work as organizers of the seminar. The authors are also grateful to Jeremy Rouse for helpful discussions concerning the main theorem, and for bringing to their attention the theorem mentioned in Remark \ref{JonesRouse}.

\section{Preliminaries}\label{preliminaries}
\subsection{Galois Representations}
Let $E/\Q$ be an elliptic curve. For any positive integer $n$, let $E[n]$ be the $n$-torsion subgroup of $E(\overline{\Q})$. It is known that $E[n]$ is a free $\Z/n\Z$-module of rank 2. We define the Tate module $T(E) \coloneqq \varprojlim_n E[n]$. Likewise, $T(E)$ is a free $\widehat{\Z}$-module of rank $2$.

The absolute Galois group $G_\Q \coloneqq \Gal(\overline{\Q}/\Q)$ naturally acts on $E[n]$ and $T(E)$. Choosing a $\widehat{\Z}$-basis of $T(E)$, equivalently a compatible system of bases of $E[n]$, we define the \textit{adelic Galois representation} of $E$ as follows:
$$\rho_E : G_\Q \to \Aut(T(E)) \simeq \GL_2(\widehat{\Z}).$$
Composing $\rho_E$ with the natural projection map $\pi_n : \GL_2(\widehat{\Z}) \to \GL_2(\Z/n\Z)$, we define the \textit{mod $n$ Galois representation} $\rho_{E,n}$. The images of $\rho_E$ and $\rho_{E,n}$ are denoted by $G_E$ and $G_E(n)$, respectively. These subgroups are well-defined up to conjugation.

When $E$ is non-CM, Serre proved the following fundamental result, known as the Serre's open image theorem.
\begin{theorem}[Serre, \protect{\cite[Th\'eor\`eme 3]{MR387283}}] \label{SOIT}
Let $E/\Q$ be a non-CM elliptic curve. Then $G_E$ is an open subgroup of $ \GL_2(\widehat{\Z})$. In particular, the adelic index $[\GL_2(\widehat{\Z}) : G_E]$ is finite.
\end{theorem}
Theorem \ref{SOIT} has the following corollary: there exists a positive integer $M$ such that
\begin{equation}\label{soitcor}
    G_E = \pi^{-1}(G_E(M))
\end{equation}
where $\pi : \GL_2(\widehat{\Z}) \to \GL_2(\Z/M\Z)$ denotes the regular projection. The smallest positive integer satisfying \eqref{soitcor} is called the adelic level of $E$, and is denoted by $m_E$. The adelic level is useful because it allows one to reduce computations involving the possibly complicated $G_E$ to computations at a finite level. Because of this usefulness, the adelic level has recently been incorporated into the LMFDB \cite{lmfdb} and has become an object of active study.

\begin{remark}
    We say that a prime $\ell$ is \textit{exceptional} for $E$ if $\rho_{E,\ell}$ is not surjective. In some literature, the term is used instead for primes $\ell$ for which the $\ell$-adic Galois representation 
    $$\rho_{E,\ell^\infty}:G_\Q\to \Aut(\varprojlim_n E[\ell^n]) \simeq \GL_2(\Z_\ell)$$
    is not surjective. For $\ell \geq 5$, these two definitions agree (see \cite[IV-23, Lemma 3]{MR263823}), while the implication fails for small primes; see \cite{elkies2006ellipticcurves3adicgalois}
    for $\ell=3$ and \cite{MR2995149} for $\ell=2$. The possible images of mod $\ell$ and $\ell$-adic Galois representations attached to elliptic curves over $\Q$ have been studied extensively; see, for example,\cite{zywina2015possibleimagesmodell,MR3482279,MR4468989}.

    The adelic level $m_E$ detects exceptional primes in the sense that if $\ell$ is exceptional, then $\ell \mid m_E$. The converse is not necessarily true, due to the phenomenon known as \textit{entanglement}; see \cite{MR4605883} for a detailed discussion of this terminology.
\end{remark}

Let $E$ be an elliptic curve of conductor $N_E$ and $p \nmid N_E$ be a rational prime. Let $\Frob_p$ denote a Frobenius element of $G_\Q$ above $p$, which is well-defined up to conjugacy; see \cite[2.1, I-6]{MR1484415} for the definition of $\Frob_p$. For any integer $m$ coprime to $p$, the mod $m$ representation satisfies
\begin{equation}\label{wellknownequation}
    \det(\rho_{E,m}(\Frob_p)) \equiv p \pmod m
    \quad \text{and} \quad
    \Tr(\rho_{E,m}(\Frob_p)) \equiv a_p(E) \pmod m.
\end{equation}
(See \cite[Chapter III. Proposition 8.6]{MR2514094}.) Let $\chi:G_\Q\to \widehat{\Z}^{\times}$ be the adelic cyclotomic character. By the Weil pairing, we have
\begin{equation}\label{WeilPairing}
    \det\circ \rho_E = \chi.
\end{equation}
Consequently, the adelic image $G_E$ has full determinant. In particular, for every positive integer $m$,
$$\det G_E(m) = (\Z/m\Z)^\times.$$

We introduce a useful lemma regarding the adelic level and conclude the section.
\begin{lemma}\label{CRTforEllipticCurve}
    Let $E/\Q$ be a non-CM elliptic curve of adelic level $m_E$. For any positive integers $m_1$ and $m_2$ with $m_1 \mid m_E^\infty$ and $\gcd(m_2,m_E) = 1$, we have
    $$G_E(m_1m_2) \simeq G_E(m_1) \times \GL_2(\Z/m_2\Z).$$
\end{lemma}
\begin{proof}
    See \cite[Lemma 2.2]{Lee_Mayle_Wang_2025}.
\end{proof}

\subsection{Serre curve}\label{SerreCurve}
A natural question is whether there exists an elliptic curve $E/\Q$ for which the adelic Galois representation $\rho_E$ is surjective. Serre \cite{MR387283} observed that this is impossible, and we explain the idea here.

Suppose that $E$ is given by a factored Weierstrass equation
$$Y^2=(X-e_1)(X-e_2)(X-e_3).$$
Since $E/\Q$ is smooth and defined over $\Q$, the roots $e_1,e_2,e_3$ lie in $\overline{\Q}$ and are distinct. The $2$-torsion module of $E$ is equal to
$$E[2]=\{\mathcal{O},(e_1,0),(e_2,0),(e_3,0)\}\simeq \Z/2\Z\times \Z/2\Z,$$
where $\mathcal{O}$ denotes the identity element. After identifying the three nontrivial $2$-torsion points with the roots $e_1,e_2,e_3$, we regard $\Aut(E[2])$ as $S_3$, the symmetric group on three elements. The discriminant $\Delta_E$ is given by
\begin{equation}\label{Discriminant}
    \Delta_E \coloneqq 16[(e_1-e_2)(e_2-e_3)(e_3-e_1)]^2.
\end{equation}
It is known that if $\Delta_E\in(\Q^\times)^2$, then the adelic index is at least 12 \cite[Proposition 2.14]{MR4732685}. Hence, we restrict our attention to the case where $\Delta_E$ is not a square.

Let $\Delta'_E$ denote the squarefree part of the discriminant; that is, $\Delta'_E$ is the unique squarefree integer such that $\Delta_E/\Delta'_E$ is the largest square. Unlike $\Delta_E$, which depends on the Weierstrass model, $\Delta'_E$ is independent of the choice of the model.

Using the identification $\Aut(E[2])\simeq S_3$, for each $\sigma\in G_\Q$, we regard $\rho_{E,2}(\sigma)$ as a permutation of the roots $e_1,e_2,e_3$. Let $\epsilon:S_3\to\{\pm1\}$ be the sign character. Then
\begin{equation}\label{mod2map}
    \sigma(\sqrt{\Delta_E}) = \epsilon(\rho_{E,2}(\sigma))\sqrt{\Delta_E}.
\end{equation}
By the Kronecker--Weber theorem, there exists the smallest positive integer $d_E$, often called a \textit{conductor} of the field $\Q(\sqrt{\Delta_E})$, such that
$$\Q(\sqrt{\Delta_E})\subseteq \Q(\zeta_{d_E}).$$
Classical algebraic number theory gives
$$d_E = \begin{cases}
|\Delta'_E| & \text{if } \Delta'_E\equiv 1 \pmod 4,\\
4|\Delta'_E| & \text{otherwise}.
\end{cases}$$
Let $\chi_{d_E}:G_\Q\to(\Z/d_E\Z)^\times$ be the cyclotomic character. Then there exists a unique quadratic character $\alpha:(\Z/d_E\Z)^\times\to\{\pm1\}$ such that, for every $\sigma\in G_\Q$,
\begin{equation}\label{cyclotomic}
    \sigma(\sqrt{\Delta_E}) = \alpha (\chi_{d_E}(\sigma)) \sqrt{\Delta_E} = (\alpha \circ \det)(\rho_{E,d_E}(\sigma)) \sqrt{\Delta_E}.
\end{equation}
Here the second equality follows from \eqref{WeilPairing}. Comparing \eqref{mod2map} and \eqref{cyclotomic}, we obtain
\begin{equation}\label{Serreentanglement}
    \epsilon (\rho_{E,2}(\sigma)) = (\alpha \circ \det)(\rho_{E,d_E}(\sigma)).
\end{equation}
Let $M_E\coloneqq \operatorname{lcm}(2,d_E)$. Define
\begin{equation}\label{HeMe}
    H_E(M_E) \coloneqq \{ M \in \GL_2(\Z/M_E\Z ): \epsilon(M \neghs \pmod 2) = (\alpha \circ \det)(M \neghs \pmod {d_E})\}.
\end{equation}
One checks that $H_E(M_E)$ is an index $2$ subgroup of $\GL_2(\Z/M_E\Z)$. Let $\pi : \GL_2(\widehat{\Z}) \to \GL_2(\Z/M_E\Z)$ be the natural projection. Define
\begin{equation}\label{He}
    H_E \coloneqq \pi^{-1}(H_E(M_E)).
\end{equation}
It follows from \eqref{Serreentanglement} that $G_E \subseteq H_E$. Therefore, the adelic index of any non-CM elliptic curve over $\Q$ is at least two. We call a non-CM elliptic curve $E/\Q$ a \textit{Serre curve} if $G_E=H_E$.

\begin{remark}
    The phenomenon in \eqref{Serreentanglement} can also be described from a field-theoretic perspective. We just observed that the quadratic field $\Q(\sqrt{\Delta_E})$ is contained both in $\Q(E[2])$ and in the cyclotomic field $\Q(\zeta_{d_E})$, which is contained in $\Q(E[d_E])$ by the Weil pairing. Consequently, the mod $M_E$ Galois representation cannot be surjective, since its mod $2$ and mod $d_E$ components must be compatible with this common quadratic subfield. This phenomenon is called a Serre entanglement.
\end{remark}

For a Serre curve $E/\Q$, the adelic level $m_E$ and the mod $m$ Galois representations are well known.

\begin{proposition}\label{sizeofmE}
    Let $E/\Q$ be a Serre curve with discriminant $\Delta_E$. Let $\Delta'_E$ denote the squarefree part of $\Delta_E$. Then
    $$m_E = \begin{cases}
        2|\Delta'_E| & \text{ if } \Delta'_E \equiv 1 \pmod 4, \\
        4|\Delta'_E| & \text{ otherwise}.
    \end{cases}$$
\end{proposition}
\begin{proof}
    See \cite[pp.696-697]{MR2534114}.
\end{proof}
\begin{lemma}\label{atleast6}
    Let $E/\Q$ be a Serre curve. Then $m_E \geq 6$.
\end{lemma}
\begin{proof}
    See \cite[Lemma 12]{hamakiotes2026densitycoprimereductionselliptic}.
\end{proof}
Recall the definition of $H_E$ in \eqref{He}. For a positive integer $m$, let $H_E(m)$ denote the image of $H_E$ under the natural projection $\GL_2(\widehat{\Z})\to \GL_2(\Z/m\Z)$.
\begin{proposition}\label{Serremodmimage}
    Let $E/\Q$ be a Serre curve of adelic level $m_E$. For a positive integer $m$, we have
    $$G_E(m) = \begin{cases}
        \GL_2(\Z/m\Z) & \text{ if } m_E \nmid m, \\
        H_E(m) & \text{ if } m_E \mid m.
    \end{cases}$$
    In particular, $|G_E(m)| = |\GL_2(\Z/m\Z)|$ if $m_E \nmid m$, and $\frac{1}{2}|\GL_2(\Z/m\Z)|$ otherwise.
\end{proposition}
\begin{proof}
    See \cite[Proposition 2.4]{Lee_Mayle_Wang_2025}.
\end{proof}
Let $E/\Q$ be a Serre curve of adelic level $m_E$. For a positive integer $m$ satisfying $m_E \mid m \mid m_E^\infty$, Jones \cite[Section 4]{MR2534114} described $H_E(m)$ as the kernel of a quadratic character on $\GL_2(\Z/m\Z)$. We recall this description below, since it will be used to compute the $m$-divisibility density $C_E^{m\divi}$ for Serre curves.

Let $A \in \GL_2(\Z/m\Z)$, and let $d$ be a divisor of $m$. We denote the reduction of $A$ modulo $d$ by $A_d$. Recall that, under the identification $\Aut(E[2]) \simeq S_3$, we defined the sign character $\epsilon : G_E(2) \to \{\pm 1\}$. We extend the domain of $\epsilon$ to $\GL_2(\mathbb Z/2^\alpha\mathbb Z)$ by setting $\epsilon(A) \coloneqq \epsilon(A_2)$ for $A\in \GL_2(\mathbb Z/2^\alpha\mathbb Z)$.

For $\alpha \geq 2$, define $\chi_4 : \GL_2(\Z/2^\alpha \Z) \to \{\pm1\}$ by
\begin{equation}\label{chi4}
    \chi_4(A) \coloneqq \begin{cases}
    1 & \text{ if } \det A_4 \equiv 1 \pmod 4, \\
    -1 & \text{ if } \det A_4 \equiv -1 \pmod 4.
\end{cases}
\end{equation}
For $\alpha \geq 3$, define $\chi_8 : \GL_2(\Z/2^\alpha \Z) \to \{\pm 1\}$ by
\begin{equation}\label{chi8}
    \chi_8(A) \coloneqq \begin{cases}
        1 & \text{ if } \det A_8 \equiv \pm 1 \pmod 8, \\
        -1 & \text{ if } \det A_8 \equiv \pm 3 \pmod 8.
    \end{cases}
\end{equation}
Now, we define a quadratic character $\psi_m:\GL_2(\mathbb Z/m\mathbb Z)\to \{\pm1\}$ by
$$\psi_m(A)=\prod_{\ell^\alpha\parallel m}\psi_{\ell^\alpha}(A_{\ell^\alpha}),$$
where
\begin{equation}\label{localcharacter}
    \psi_{\ell^\alpha} (A) \coloneqq \begin{cases}
    \left(\frac{\det A_\ell}{\ell}\right)& \text{ if } \ell \text{ is an odd prime,}\\
    \epsilon(A) & \text{ if } \ell =2 \text{ and } \Delta'_E \equiv 1 \pmod 4,\\
    \chi_4(A)\epsilon(A) &\text{ if } \ell = 2 \text{ and } \Delta'_E \equiv 3 \pmod 4, \\
    \chi_8(A) \epsilon(A) & \text{ if }\ell = 2\text{ and } \Delta'_E \equiv 2\pmod 8, \\
    \chi_8(A) \chi_4(A) \epsilon(A) & \text{ if } \ell = 2 \text{ and } \Delta'_E \equiv 6 \pmod 8,
\end{cases}
\end{equation}
where $\left(\frac{\cdot}{\ell}\right)$ denotes the Legendre symbol. In case $m_E \mid m \mid m_E^\infty$, Jones proved that
\begin{equation}\label{kerofpsim}
    H_E(m) = \ker \psi_m.
\end{equation}

\subsection{Primes of $m$-divisibility reduction for $E$}\label{mdivisibilityprime}
In this section, we interpret the condition for the $m$-divisibility prime for $E$ in terms of Galois representation. Also, we evaluate the density, given that $G_E(m)$ is the full group.

Let $p \nmid mN_E$ be a rational prime. By \eqref{wellknownequation}, we see that $p$ is a prime of $m$-divisibility prime for $E$ if and only if
$$\det(I-\rho_{E,m}(\Frob_p)) \equiv p+1-a_p(E) \equiv 0 \pmod m.$$
We define
\begin{equation}\label{Psim}
    \Psi(m) \coloneqq \{M \in \GL_2(\Z/m\Z) : \det(I-M) \equiv 0 \pmod m\}.
\end{equation}
One can check that $G_E(m) \cap \Psi(m)$ is closed under conjugation. Thus, by the Chebotarev density theorem, we have
\begin{equation}\label{CmdivE}
    C_E^{m\divi} = \frac{|G_E(m) \cap \Psi(m)|}{|G_E(m)|}.
\end{equation}
In the following remarks, we discuss the upper and lower bounds of $C_E^{m\divi}$.
\begin{remark}\label{upperbound}
    Suppose that $E(\Q)$ contains a torsion point of order $m$. For every prime $p$ of good reduction, the mod $p$ reduction induces an injection $E(\Q)_{\text{tors}} \hookrightarrow E_p(\F_p)$ (see \cite[VII. Proposition 3.1]{MR2514094}). Thus $m \mid \#E_p(\F_p)$ for every good prime $p$ and $C_E^{m\divi}=1$. By Mazur's classification theorem \cite{MR482230}, this can occur only for $m \leq 10$, $m = 12$, or $m =16$.
\end{remark}

\begin{remark}\label{lowerbound}
    Let $p \nmid mN_E$ and suppose that $\rho_{E,m}(\Frob_p) = I$. This is equivalent to saying that $p$ splits completely in $\Q(E[m])$. Hence, the Frobenius endomorphism $\pi_p$ fixes every point of $E[m]$. Therefore, for such primes, $E[m] \subseteq E_p(\F_p)$, and hence $|E_p(\F_p)|$ is divisible by $m^2$. In particular, $p$ is an $m$-divisibility prime.

    Since the identity element forms a conjugacy class by itself, the Chebotarev density theorem guarantees that
    $$C_E^{m\divi} \geq \frac{1}{|G_E(m)|} \geq \frac{1}{|\GL_2(\Z/m\Z)|}.$$
    Therefore, for any positive integer $m$, the density of the $m$-divisibility primes is non-zero.
\end{remark}

\begin{lemma}\label{mixedgroupdecomposition}
    Let $E/\Q$ be a non-CM elliptic curve of adelic level $m_E$ and $m$ be a positive integer. Write $m = m_1m_2$ so that $m_1 = \gcd(m,m_E^\infty)$ and $\gcd(m_2,m_E) = 1$. Then we have
    $$\frac{|G_E(m) \cap \Psi(m)|}{|G_E(m)|} = \frac{|G_E(m_1) \cap \Psi(m_1)|}{|G_E(m_1)|} \prod_{\ell^\alpha \parallel m_2} \frac{|\Psi(\ell^\alpha)|}{|\GL_2(\Z/\ell^\alpha\Z)|}.$$
\end{lemma}

\begin{proof}
By Proposition \ref{CRTforEllipticCurve}, we have an isomorphism $\iota : G_E(m) \to G_E(m_1) \times \GL_2(\Z/m_2\Z)$. Under this isomorphism, since $m_1$ and $m_2$ are coprime, the condition defining $\Psi(m)$ splits into the corresponding conditions modulo $m_1$ and $m_2$. Indeed, if $\iota(M) = (M_1,M_2)$, then 
$$\det(I-M) \equiv 0 \pmod {m_1m_2} \iff \det(I-M_1) \equiv 0 \pmod {m_1} \quad \text{ and } \quad \det(I-M_2) \equiv 0 \pmod {m_2},$$
and hence $\iota$ induces a one-to-one correspondence between
$$\iota : G_E(m) \cap \Psi(m) \longleftrightarrow \left(G_E(m_1) \cap \Psi(m_1)\right) \times \Psi(m_2).$$
Finally, one can check that the isomorphism induces an one-to-one correspondence between
$$\Psi(m_2) \longleftrightarrow \prod_{\ell^\alpha \parallel m_2} \Psi(\ell^\alpha).$$
This completes the proof.
\end{proof}

\begin{lemma}\label{FullGroupComputation}
    Let $\ell$ be a prime and $\alpha \geq 1$ be an integer. Then
    $$\frac{|\Psi(\ell^\alpha)|}{|\GL_2(\Z/\ell^\alpha\Z)|} =\frac{1}{\varphi(\ell^\alpha)}
        \left(1-\frac{1}{\varphi(\ell^\alpha)(\ell+1)}\right).$$
\end{lemma}
Before proceeding with the proof, we briefly explain the strategy. Rather than counting the elements of $\Psi(\ell^\alpha)$ directly, we lift the problem to $\GL_2(\Z_\ell)$ and compute the Haar measure of the corresponding subset. The desired proportion agrees with the Haar measure of its full preimage in $\GL_2(\Z_\ell)$. We then use the results of Lombardo--Perucca \cite{MR3690236} on the distribution of the $1$-eigenspace of matrices in $\GL_2(\Z_\ell)$.

\begin{proof}[Proof of Lemma \ref{FullGroupComputation}]
    Let $G = \GL_2(\Z_\ell)$ and $G(\ell^\alpha) = \GL_2(\Z/\ell^\alpha\Z)$. We denote the natural projection $G \to G(\ell^\alpha)$ by $\pi$, and let $\mu$ be the normalized Haar measure on $G$. Each fiber of $\pi$ has measure $1/|G(\ell^\alpha)|$, and hence
    $$\frac{|\Psi(\ell^\alpha)|}{|G(\ell^\alpha)|} = \mu ( \pi^{-1}(\Psi(\ell^\alpha))).$$
    Additionally, we define
    $$V_\ell \coloneqq \varinjlim_n(\Z/\ell^n\Z)^2.$$
    Following the notation of \cite[Section 3.2]{MR3690236}, we view each
    element of $G$ as an automorphism of $V_\ell$. For $A\in G$, its $1$-eigenspace is
    $$\ker(I-A\mid V_\ell) = \{v\in V_\ell:Av=v\}.$$
    For $a,b\geq 0$, we set
    $$\mathcal M_{a,b} \coloneqq \{A\in G: \ker(I-A\mid V_\ell) \simeq \Z/\ell^a\Z\times \Z/\ell^{a+b}\Z\}.$$
    By \cite[Lemma 23]{MR3690236}, if $A\in\mathcal M_{a,b}$, then $v_\ell(\det(I-A)) = 2a+b$. Thus, for $A\in\mathcal M_{a,b}$, we have
    $$A \in \pi^{-1}(\Psi(\ell^\alpha)) \iff 2a+b \geq \alpha.$$
    Observe that the set of matrices with infinite $1$-eigenspace is contained in $\pi^{-1}(\Psi(\ell^\alpha))$ for any $\alpha \geq 1$. However, \cite[Lemma 25]{MR3690236} states that such set has measure zero, and hence
    $$\frac{|\Psi(\ell^\alpha)|}{|\GL_2(\Z/\ell^\alpha \Z)|} = \mu(\pi^{-1}(\Psi(\ell^\alpha))) =\sum_{2a+b\geq \alpha}\mu_{a,b},$$
    where $\mu_{a,b}=\mu(\mathcal M_{a,b})$. We put $c = \lceil \alpha/2\rceil$ and $d = \lfloor \alpha/2\rfloor$. Then we have
    $$\frac{|\Psi(\ell^\alpha)|}{|\GL_2(\Z/\ell^\alpha \Z)|} = \sum_{b\geq \alpha} \mu_{0,b} + \sum_{a \geq c} \mu_{a,0} + \sum_{\substack{a \geq 1 \\ b\geq 1 \\ 2a+b \geq \alpha}}\mu_{a,b}.$$
    Applying \cite[Theorem 2]{MR3690236}, we now evaluate each series. For the first sum, we have
    $$\sum_{b\geq \alpha} \mu_{0,b} = \frac{\ell^2-\ell-1}{\ell(\ell-1)} \sum_{b \geq \alpha} \ell^{-b} = \frac{\ell^2-\ell-1}{\ell^\alpha(\ell-1)^2}.$$
    For the second sum, we have
    $$\sum_{a\geq c} \mu_{a,0} = \sum_{a \geq c}  \ell^{-4a} = \frac{\ell^{4-4c}}{\ell^4-1}.$$
    For the last sum, suppose that $a,b \geq 1$. The condition $2a+b\geq \alpha$ is equivalent to $b \ge B(a) \coloneqq \max\{1,\alpha-2a\}$. If $\alpha \geq 2$, then
    $$B(a)= \begin{cases}
        \alpha-2a & \text{ if } 1 \leq a \leq d-1, \\
        1 & \text{ if } a \geq d
    \end{cases}.$$
Thus, we have
\begin{align*}
\sum_{\substack{a \geq 1 \\ b \geq 1 \\ 2a+b \geq \alpha}} \mu_{a,b}&= \sum_{a\geq 1}\sum_{b\geq B(a)}\mu_{a,b} = \sum_{a\geq 1}(\ell+1)\ell^{-4a-1}\sum_{b\geq B(a)}\ell^{-b} \\
&= \frac{\ell+1}{\ell-1}\left(\sum_{1\leq a\leq d-1}\ell^{-2a-\alpha}+\sum_{a\geq d}\ell^{-4a-1}\right) =\frac{1-\ell^{2-2d}}{\ell^\alpha(\ell-1)^2}+\frac{\ell^{3-4d}}{(\ell-1)^2(\ell^2+1)}.
\end{align*}
For $\alpha=1$, we have $(B(a)=1$ for all $a\geq 1$, and the same formula is obtained by direct summation.
Combining all these series, we obtain
$$\frac{|\Psi(\ell^\alpha)|}{|\GL_2(\Z/\ell^\alpha\Z)|} = \frac{\ell^2-\ell-1}{\ell^\alpha(\ell-1)^2} + \frac{\ell^{4-4c}}{\ell^4-1} +\frac{1-\ell^{2-2d}}{\ell^\alpha(\ell-1)^2} + \frac{\ell^{3-4d}}{(\ell-1)^2(\ell^2+1)}.$$
A direct calculation, according as $\alpha$ is even or odd, gives the desired formula.
\end{proof}
By Lemma \ref{mixedgroupdecomposition}, Lemma \ref{FullGroupComputation}, and the multiplicativity of $\varphi$, we deduce the following result.
\begin{proposition}\label{FullGroupCase}
    Let $E$ be an elliptic curve and $m$ be a positive integer. Suppose that $G_E(m) \simeq \GL_2(\Z/m\Z)$. Then we have
    $$C_E^{m\divi} = \frac{1}{\varphi(m)} \prod_{\ell^\alpha \parallel m}
        \left(1-\frac{1}{\varphi(\ell^\alpha)(\ell+1)}\right).$$
\end{proposition}

\subsection{Average Density}\label{Average Density}
In this section, we explicitly calculate the average $m$-divisibility density $C^{m\divi}$ proposed by Banks and Shparlinski \cite{MR2570668} and show that it is equal to $C_E^{m\divi}$ given that $G_E(m) \simeq \GL_2(\Z/m\Z)$. First we revisit the notations introduced in their paper. Define a multiplicative function $\omega_q$ defined on prime powers as follows:
\begin{equation}\label{omegaq}
    \omega_q(\ell^j) \coloneqq \begin{cases}
    \dfrac{1}{\ell^{j-1}(\ell-1)} & \text{ if } q \not \equiv 1 \pmod {\ell^{\lceil j/2 \rceil}}, \\
    \\
    \dfrac{\ell^{\lfloor j/2\rfloor +1} + \ell^{\lfloor j/2 \rfloor} - 1}{\ell^{j+\lfloor j/2\rfloor -1}(\ell^2-1)} & \text{ if } q \equiv 1 \pmod {\ell^{\lceil j/2 \rceil}}.
\end{cases}
\end{equation}
For a positive integer $m$, we define
$$\overline{m} \coloneqq \prod_{\ell^j \parallel m} \ell^{\lceil j/2 \rceil}.$$
Then the average density $C^{m\divi}$ is given by
\begin{equation}\label{definitionofaveragedensity}
    C^{m\divi} = \frac{1}{\varphi(\overline{m})}\sum_{\substack{1 \leq q \leq \overline{m} \\ (q,\overline{m})=1}} \omega_q(m).
\end{equation}

\begin{remark}\label{HoweWork}
We briefly explain Howe's theorem and the idea behind the average density $C^{m\divi}$. Let $q$ be a prime power, and let $V(\F_q)$ be the set of $\F_q$-isomorphism classes of elliptic curves over $\F_q$. For a positive integer $m$, set
$$V(\F_q;m)\coloneqq \{E/\F_q : m \mid \#E(\F_q)\}/\simeq_{\F_q}.$$
For a subset $S\subseteq V(\F_q)$, define its weighted cardinality by
$$\#'S \coloneqq \sum_{[E]\in S} \frac{1}{\#\Aut_{\F_q}(E)}.$$
Howe \cite{MR1204781} proved that $\#'V(\F_q)=q$. Thus, it is natural to interpret $\#'V(\F_q;m)/q$ as the probability that a randomly chosen elliptic curve $E$ over $\F_q$ has $\#E(\F_q)$ divisible by $m$. He further proved that
$$\left|\frac{\#'V(\F_q;m)}{q} - \omega_q(m)\right| \ll \frac{m^{1+o(1)}}{\sqrt{q}},$$
with an effective bound. In addition, Howe observed that $\omega_q(m)>1/m$ for every integer $m\geq 2$. Thus, for sufficiently large $q$, the order of a random elliptic curve over $\F_q$ is more likely to be divisible by $m$ than a random integer is. This phenomenon is in line with Corollary \ref{maincorollary}.

Banks and Shparlinski \cite{MR2570668} used Howe's finite field estimate as the local input for an average result over the family $\mathcal{F}$. Roughly speaking, they proved that, after averaging over this family, the reductions of curves $E/\Q$ modulo primes $p$ are sufficiently well distributed among the relevant $\F_p$-isomorphism classes. This allows the average behavior of the condition $m\mid \#E_p(\F_p)$ to be described in terms of Howe's local quantities $\omega_q(\ell^j)$, averaged over the relevant residue classes of $q$.
\end{remark}
Now, we prove that the average density $C^{m\divi}$ is equal to $C^{m\divi}_E$ when $G_E(m) \simeq \GL_2(\Z/m\Z)$. For this purpose, we first prove that $C^{m\divi}$ is multiplicative, and those two densities coincide when at prime powers.
\begin{lemma}\label{multiplicativityofaveragedensity}
    Let $m = \ell^\alpha m'$ where $\gcd(\ell,m') = 1$. Then we have
    $$C^{m\divi} = C^{\ell^\alpha\divi} \cdot C^{m' \divi}.$$
\end{lemma}
\begin{proof}
   Since $\ell^\alpha$ and $m'$ are coprime, we have $\overline{m} = \overline{\ell^\alpha} \cdot \overline{ m'}$. Recall that $\omega_q(m)$ is determined by the congruence class of $q \pmod {\overline{m}}$. By the Chinese remainder theorem, we have $(\Z/\overline{m}\Z)^\times \simeq (\Z/\overline{\ell^\alpha}\Z)^\times \times (\Z/\overline{m'}\Z)^\times$ and $\varphi(\overline{m}) = \varphi(\overline{\ell^\alpha})\varphi(\overline{m'})$. Thus, we have
   $$\omega_q(m) = \omega_q(\ell^\alpha m') = \omega_q(\ell^\alpha) \cdot \omega_{q}(m') = \omega_{q \neghs \pmod {\overline{\ell^\alpha}}}(\ell^\alpha) \cdot \omega_{q \neghs \pmod {\overline{m'}}}(m'),$$
   and hence
    $$\frac{1}{\varphi(\overline{m})}\sum_{\substack{1 \leq q \leq \overline{m} \\ (q,\overline{m}) = 1}} \omega_q(m) = \left(\frac{1}{\varphi(\overline{\ell^\alpha})}\sum_{\substack{1 \leq q_1 \leq \overline{\ell^\alpha} \\ (q_1,\overline{\ell^\alpha}) = 1}} \omega_{q_1}(\ell^\alpha) \right) \left(\frac{1}{\varphi(\overline{m'})}\sum_{\substack{1 \leq q_2 \leq \overline{m'} \\ (q_2,\overline{m'}) = 1}} \omega_{q_2}(m')\right).$$
    This completes the proof.
\end{proof}
\begin{lemma}\label{averagedensitycomputation}
    Let $\ell$ be a prime and $\alpha \geq 1$ be an integer. Then
$$C^{\ell^\alpha \divi}=\frac{1}{\varphi(\ell^\alpha)}
        \left(1-\frac{1}{\varphi(\ell^\alpha)(\ell+1)}\right).$$
\end{lemma}

\begin{proof}
    Let $c = \lfloor \alpha/2\rfloor$. By \eqref{omegaq}, we have
    \begin{equation}\label{intermediatestep}
        \sum_{\substack{1\le q\le \ell^{\alpha-c}\\(q,\overline{\ell^\alpha})=1}}\omega_q(\ell^\alpha)= \frac{\varphi(\ell^{\alpha-c})-1}{\ell^{\alpha-1}(\ell-1)} + \frac{\ell^{c+1}+\ell^c-1}{\ell^{\alpha+c-1}(\ell^2-1)} = \frac{\varphi(\ell^{\alpha-c})}{\varphi(\ell^\alpha)} - \frac{1}{\varphi(\ell^\alpha)\ell^c(\ell+1)}.
    \end{equation}
    Note that $\varphi(\ell^{\alpha-c})\ell^c = \varphi(\ell^\alpha)$. Therefore, multiplying both sides of \eqref{intermediatestep} by
    $1/\varphi(\ell^{\alpha-c})$, we obtain the desired result.
\end{proof}
By Proposition \ref{FullGroupCase}, Lemma \ref{multiplicativityofaveragedensity}, and Lemma \ref{averagedensitycomputation}, we deduce the following result:
\begin{proposition}\label{fullgroupcaseofconstant}
    Let $m$ be a positive integer. Then
      $$C_E^{m\divi} = \frac{1}{\varphi(m)} \prod_{\ell^\alpha \parallel m}
        \left(1-\frac{1}{\varphi(\ell^\alpha)(\ell+1)}\right).$$
    Given an elliptic curve $E$ with $G_E(m) \simeq \GL_2(\Z/m\Z)$, we have $C^{m\divi} = C_E^{m\divi}$.
\end{proposition}

\begin{remark}
Proposition \ref{fullgroupcaseofconstant} shows that the average density $C^{m\divi}$ agrees with the density obtained from a single elliptic curve whose mod $m$ Galois representation has full image. In this sense, one may view $C^{m\divi}$ as the density $C_E^{m\divi}$ that would arise if the adelic Galois representation of $E$ were surjective. As discussed in Section \ref{SerreCurve}, no elliptic curve over $\Q$ has surjective adelic Galois representation. Nevertheless, this heuristic provides a useful guide for predicting average densities in related problems.

For the cyclicity problem, Banks and Shparlinski \cite{MR2570668} studied the average density over families of elliptic curves, while Jones \cite{MR2534114} expressed the density for a fixed non-CM elliptic curve in terms of its adelic Galois image. Formally replacing the adelic image by the full adelic group in Jones's formula recovers the average density. For the Koblitz--Zywina conjecture, Balog, Cojocaru, and David \cite{MR2843097} studied the average asymptotic behavior over families of elliptic curves, while Zywina \cite{MR2805578} expressed the constant for a fixed elliptic curve in terms of its adelic Galois image. Again, formally replacing the adelic image by the full adelic group in Zywina's formula recovers the average behavior studied by Balog--Cojocaru--David. Thus, in each of these problems, the average behavior is governed by the generic Galois image, whereas the behavior of a fixed elliptic curve may contain correction factors coming from its actual adelic image.
\end{remark}

\section{Computing the density for Serre curve} \label{S:index2density}
In the previous section, we computed the $m$-divisibility density $C_E^{m\divi}$ in the case where $G_E(m) \simeq \GL_2(\Z/m\Z)$. In this section, we compute $C_E^{m\divi}$ for a Serre curve in the case where $m_E \mid m$.

Let $m = m_1m_2$ with $m_1 = \gcd(m,m_E^\infty)$ and $\gcd(m_2,m_E) = 1$. By Lemma \ref{mixedgroupdecomposition} and Proposition \ref{fullgroupcaseofconstant}, we have
\begin{equation}\label{correctdecomposition}
    C_E^{m\divi} = \frac{|G_E(m_1) \cap \Psi(m_1)|}{|G_E(m_1)|} \cdot C^{m_2\divi}.
\end{equation}
It follows from Proposition \ref{Serremodmimage}, \eqref{kerofpsim}, and \eqref{Psim} that
$$|G_E(m_2)| = \frac{1}{2}|\GL_2(\Z/m_2\Z)| \quad \text{ and } \quad |G_E(m_1) \cap \Psi(m_1)| = |\ker \psi_{m_1} \cap \Psi(m_1)|.$$
Additionally, as shown in the proof of Lemma \ref{mixedgroupdecomposition}, the isomorphism induced by the Chinese remainder theorem gives a one-to-one correspondence
$$\Psi(m) \longleftrightarrow \prod_{\ell^\alpha \parallel m} \Psi(\ell^\alpha).$$
In particular, $\Psi(m)$ satisfies the hypothesis of \cite[Lemma 16]{MR2534114}, and this gives the following proposition.
\begin{proposition}\label{countingprop}
    Let $E/\Q$ be a Serre curve of adelic level $m_E$ and $m$ be a positive integer such that $m_E \mid m \mid m_E^\infty$. Then we have
    $$|\psi_m^{-1}(+1) \cap \Psi(m)| = \frac{1}{2}\left(|\Psi(m)|  + \prod_{\ell^\alpha \parallel m}\left(|\psi^{-1}_{\ell^\alpha} (+1) \cap \Psi(\ell^\alpha)| - |\psi^{-1}_{\ell^\alpha}(-1) \cap \Psi(\ell^\alpha)|\right) \right).$$
\end{proposition}
For notational convenience, we express
$$Y_{\ell^\alpha, \pm} \coloneqq \psi_{\ell^\alpha}^{-1}(\pm 1) \cap \Psi(\ell^\alpha).$$
It suffices to count $|Y_{\ell^\alpha, \pm}|$ for each prime powers $\ell^\alpha$. We first introduce some notation and lemmas. As one can see from \eqref{localcharacter}, $\psi_{\ell^\alpha}$ is defined differently when $\ell = 2$, so we treat this case separately.

\subsection{Prerequisites for $\ell = 2$ case}
First of all, we introduce an identity that will be used frequently throughout the section. For any commutative ring $R$ and $A,B \in M_2(R)$, we have
\begin{equation}\label{eq:det_expand_2x2}
\det(A+B)=\det(A)+\det(B)+\Tr(\adj(A)B),
\end{equation}
where $\adj(A)$ denotes the adjugate matrix of $A$. Additionally, throughout this section and henceforth, we will identify $2^\alpha\Z/2^{\alpha+\beta}\Z$ with $\Z/2^\beta\Z$ in the natural way without further comment.

Let $A \in M_2(\Z/2^\alpha \Z)$ with $\det A = 0$ and $A \equiv 0 \pmod 2$. Let $B$ be any lifting of $A$ modulo $2^{\alpha+1}$. We define $\delta(A) \in \Z/2\Z$ so that
\begin{equation}\label{defofdelta}
    \det B\equiv 2^\alpha \delta (A) \pmod {2^{\alpha+1}}.
\end{equation}
\begin{lemma}
    The map $\delta$ is well-defined.
\end{lemma}
\begin{proof}
    Let $B_1$ and $B_2$ be two such lifts of $A$. Then there exists $X \in M_2(\Z/2\Z)$ such that 
    $$B_1 = B_2 + 2^\alpha X.$$
    By \eqref{eq:det_expand_2x2}, we have
    $$\det B_1 \equiv \det B_2 + 2^\alpha \Tr (\adj(B_2)X) \pmod {2^{\alpha+1}}.$$
    Since $A \equiv 0 \pmod 2$, we have $\adj(B_2) \equiv 0 \pmod 2$. Thus, we have
    $$\det B_1 \equiv \det B_2 \pmod {2^{\alpha+1}}.$$
    Hence, $\delta$ does not depend on the choice of the lift, given that $\det A = 0$ and $A \equiv 0 \pmod 2$.
\end{proof}

We define the following sets that will be used throughout:
\begin{equation}\label{WalphaValpha}
    \begin{aligned}
        W_\alpha &\coloneqq \{X \in M_2(\Z/2^\alpha\Z) : \det X \equiv 0  \pmod {2^{\alpha-1}}\} \\
        V_\alpha &\coloneqq \{X \in M_2(\Z/2^\alpha\Z) : \det X \equiv 0 \pmod {2^{\alpha}}\}
    \end{aligned}
\end{equation}
Clearly, $V_\alpha \subseteq W_\alpha$. If $\alpha = 1$,  the determinant condition in $W_1$ becomes trivial, and hence we take $M_2(\Z/2\Z)$ as $W_1$. We also use the following partitions:
\begin{equation}\label{WalphaValphapartitioned}
    \begin{aligned}
        W_\alpha &= \{X \in W_\alpha : X \equiv 0 \pmod 2\} \sqcup \{X\in W_\alpha: X \not \equiv 0 \pmod 2\} =: W_{\alpha,0} \sqcup W_{\alpha,*},\\
        V_\alpha &= \{X \in V_\alpha :  X \equiv 0 \pmod 2\} \sqcup \{X \in V_\alpha : X \not \equiv 0 \pmod 2\} =: V_{\alpha,0}\sqcup V_{\alpha,*}.
\end{aligned}
\end{equation}

\begin{lemma}\label{M2lifting}
    Let $A \in M_2(\Z/2^\alpha \Z)$ with $\det A = 0$. Then, we have
    $$\#\{B \in M_2(\Z/2^{\alpha+1}\Z ): B \equiv A \neghs \pmod {2^{\alpha}}, \det B \equiv 0 \neghs \pmod {2^{\alpha+1}}\} = \begin{cases}
        8 & \text{if } A \not \equiv 0 \neghs \pmod 2, \\
        16 & \text{if } A \equiv 0 \neghs \pmod 2 \text{ and } \delta(A) = 0, \\
        0 & \text{if } A \equiv 0 \neghs \pmod 2 \text{ and } \delta(A) = 1.
    \end{cases}$$
    In particular, for a non-zero rank 1 matrix $A \in M_2(\Z/2\Z)$, we have
    $$\#\{B \in M_2(\Z/2^{\alpha}\Z) : B \equiv A  \pmod 2, \det B \equiv 0  \pmod {2^{\alpha}}\} = 8^{\alpha-1}.$$
\end{lemma}

\begin{proof}
    Let $A$ satisfy the hypothesis in the first claim and fix a lift $B$ of $A$ modulo $2^{\alpha+1}$. Then any lift of $A$ modulo $2^{\alpha+1}$ can be uniquely expressed as $B + 2^\alpha X$ for some $X \in M_2(\Z/2\Z)$. By \eqref{eq:det_expand_2x2},
    $$\det(B+2^\alpha X) \equiv \det B + 2^\alpha \Tr(\adj(B) X ) \pmod {2^{\alpha+1}}.$$
    Let $L_B : M_2(\Z/2\Z) \to \Z/2\Z$ by $L_B(X) \coloneqq \Tr(\adj(B)X) \pmod 2$. Then, we have
        \begin{equation}\label{linearfunctionalcondition}
\det(B+2^\alpha X) \equiv 0 \pmod {2^{\alpha+1}} \iff \det B + 2^\alpha L_B(X) \equiv 0 \pmod {2^{\alpha+1}}.\end{equation}

    Suppose $A \equiv 0 \pmod 2$. Then $\adj(B) \equiv 0 \pmod 2$, and hence $L_B \equiv 0 \pmod 2$. By \eqref{defofdelta}, the right-hand side of \eqref{linearfunctionalcondition} turns into 
    $$2^\alpha \delta (A) \equiv 0 \pmod {2^{\alpha+1}}.$$
    All 16 choices of $X \in M_2(\Z/2\Z)$ satisfy the condition if $\delta(A) = 0$, while none does if $\delta(A) = 1$.

    Suppose $A \not \equiv 0 \pmod 2$, then $\adj(B) \not \equiv 0 \pmod 2$. Thus, $L_B$ is a nonzero linear functional, and hence surjective. Therefore, exactly half of $M_2(\Z/2\Z)$ satisfies the right-hand side of \eqref{linearfunctionalcondition}. The second claim follows from the induction.
\end{proof}

\begin{lemma}\label{Fixageq2}
    Fix $\alpha \geq 2$. The following hold:
    \begin{enumerate}[\label=(\roman*)]
        \item $|W_\alpha| = 16|V_{\alpha-1}|$;
        \item $|V_{\alpha,*}| = 9 \cdot 2^{3\alpha-3}$;
        \item $|W_{\alpha,*}| = 9 \cdot 2^{3\alpha-2}$;
        \item $|V_{\alpha,0}| = |W_{\alpha-1}|$.
    \end{enumerate}
\end{lemma}
\begin{proof}
    Part (ii) follows from Lemma \ref{M2lifting} and the fact that there are nine non-zero rank 1 matrices over $\Z/2\Z$.

    Let $\pi : M_2(\Z/2^\alpha \Z) \to M_2(\Z/2^{\alpha-1}\Z)$ be the natural projection. Let $A \in W_\alpha$. Clearly, $\pi(A) \in V_{\alpha-1}$. Fix $B_0 \in V_{\alpha-1}$ and let $B$ be any lift of $B_0$ modulo $2^\alpha$. Then any lift of $B_0$ can be uniquely expressed as $B+2^{\alpha-1}X$ for some $X \in M_2(\Z/2\Z)$. By \eqref{eq:det_expand_2x2}, we have
    $$\det(B+2^{\alpha-1}X) \equiv \det (B) + 2^{2\alpha-2}\det (X) + 2^{\alpha-1} \Tr(\adj(B)X) \equiv \det B \equiv \det B_0 \equiv 0 \pmod {2^{\alpha-1}},$$
    since $\alpha \geq 2$. This shows that $\pi^{-1}(V_{\alpha-1}) = W_\alpha$. Since each fiber has size 16, part (i) follows. Similarly, one can check that $\pi^{-1}(V_{\alpha-1,*}) = W_{\alpha,*}$, and hence part (iii) follows from part (ii). 

    Let $X \in V_{\alpha,0}$. Then there exists a unique $X' \in M_2(\Z/2^{\alpha-1}\Z)$ for which $X = 2X'$. Additionally,
$$\det X \equiv 0 \pmod {2^{\alpha}} \iff \det X' \equiv 0 \pmod {2^{\alpha-2}},$$
and hence it yields a one-to-one correspondence between
$$V_{\alpha,0} \longleftrightarrow \{X' \in M_2(\Z/2^{\alpha-1}\Z) : \det X' \equiv 0 \pmod {2^{\alpha-2}}\} = W_{\alpha-1}.$$
Thus, part (iv) follows.
\end{proof}

\subsection{The case $\ell = 2$}

By \eqref{Psim}, recall that
$$Y_{2^\alpha,\pm} = \{A \in \GL_2(\Z/2^\alpha\Z) : \psi_{2^{\alpha}}(A) = \pm 1 \text{ and } \det(I-A) \equiv 0 \pmod {2^{\alpha}}\}.$$
We aim to compute $|Y_{2^{\alpha},+}| - |Y_{2^{\alpha},-}|$. Since we are taking $m_E \mid m$ in this chapter, we have $\alpha \coloneqq v_2(m) \geq v_2(m_E)$ and hence
\begin{equation}\label{conditiononalpha}
    \Delta'_E \equiv \begin{cases}
        1 \pmod 4 \implies \alpha \geq 1,\\
        3 \pmod 4 \implies \alpha \geq 2,\\
        2 \pmod 4 \implies \alpha \geq 3.
        \end{cases}
\end{equation}
Fix $\beta \geq \alpha$ and $A \in \GL_2(\Z/2^{\alpha} \Z)$. We define 
$$G_{\alpha}^\beta(A) \coloneqq \{B\in \GL_2(\Z/2^\beta \Z) : B\equiv A \neghs \pmod {2^\alpha} \text{ and } \det(I-B) \equiv 0 \neghs \pmod {2^{\beta}}\}.$$
Then we have
\begin{equation}\label{Ysum}
    |Y_{2^\alpha,+}| - |Y_{2^\alpha,-}| = \sum_{A \in \GL_2(\Z/2\Z)}\sum_{B \in G_1^\alpha(A)} \psi_{2^\alpha}(B) =: \sum_{A \in \GL_2(\Z/2\Z)} n_\alpha(A).
\end{equation}
We identify $\GL_2(\Z/2\Z) \simeq S_3$. The group is composed of three conjugacy classes: the classes of the identity $I$, the transpositions $\tau$, and the 3-cycles $\sigma$.
\begin{lemma}\label{transpositionlemma}
    Let $\tau$ and $\tau'$ be two distinct transpositions in $\GL_2(\Z/2\Z)$. Then $n_\alpha(\tau) = n_\alpha(\tau')$.
\end{lemma}
\begin{proof}
Let $h \in \GL_2(\Z/2\Z)$ so that $h\tau h^{-1} = \tau'$. Let $T \in G_1^{\alpha}(\tau)$ and $H$ be a lift of $h$ modulo $2^\alpha$. Then we have
$$HTH^{-1} \equiv h\tau h^{-1} \pmod 2 \quad \text{ and } \quad \det(I-HTH^{-1}) = \det(I-T) \equiv 0 \pmod {2^{\alpha}}.$$
Thus, the conjugation by $H$ gives a bijective map $G_1^\alpha(\tau) \to G_1^\alpha(\tau')$. By \eqref{localcharacter}, we see that $\psi_{2^{\alpha}}(HTH^{-1}) = \psi_{2^\alpha}(T)$ for every $T \in G_1^\alpha(\tau)$. This completes the proof.
\end{proof}
Let $\sigma$ be a 3-cycle in $\GL_2(\Z/2\Z)$. Then, the eigenvalue of $\sigma$ does not lie in $\Z/2\Z$, and hence $\det(I-\sigma) \not \equiv 0 \pmod 2$. That is, $G_1^\alpha(\sigma)$ is empty for any $\alpha \geq 1$, and hence the conjugacy class of 3-cycles does not contribute to \eqref{Ysum}. Fixing a transposition $\tau$, by Lemma \ref{transpositionlemma}, we have
\begin{equation}\label{Yalphasum}
    |Y_{2^\alpha,+}| - |Y_{2^\alpha,-}| = n_\alpha(I) + 3n_\alpha(\tau).
\end{equation}
We first compute $n_\alpha(\tau)$.
\begin{lemma}\label{nalphatau}
    Let $E/\Q$ be a Serre curve whose square part of the discriminant is $\Delta'_E$. Choose a transposition $\tau \in \GL_2(\Z/2\Z)$. Then
    $$n_\alpha(\tau) = \begin{cases}
        -2^{3(\alpha-1)} & \text{ if } \Delta'_E \equiv 1 \pmod 4, \\
        0 & \text{ otherwise}.
    \end{cases}$$
\end{lemma}

\begin{proof}
We first assume that $\Delta'_E \equiv 1 \pmod 4$. By \eqref{localcharacter}, for every $T \in G_1^\alpha(\tau)$, we have
$$\psi_{2^{\alpha}}(T) = \epsilon(\tau) =  -1.$$
Thus, $n_\alpha(\tau) = -|G_1^\alpha(\tau)|$. Observe that there is a one-to-one bijection between
$$G_1^\alpha(\tau) \longleftrightarrow \{B \in M_2(\Z/2^\alpha \Z) : B \equiv I-\tau \pmod {2}, \det B \equiv 0 \pmod {2^{\alpha}}\}$$
by $T \mapsto I-T$. Since $I - \tau$ is a non-zero rank 1 matrix over $\Z/2\Z$, it follows from Lemma \ref{M2lifting} that $|G_1^\alpha(\tau)| = 8^{\alpha-1}$.

Secondly, we assume that $\Delta'_E \equiv 3 \pmod 4$. We first argue the case $\alpha = 2$. Let us define linear functionals $L_1, L_2 : M_2(\Z/2\Z) \to \Z/2\Z$ by
\begin{equation}\label{L1L2}
    L_1(X) \coloneqq \Tr (\adj(I-\tau)X) \quad \text{ and } \quad L_2(X) \coloneqq \Tr (\adj(\tau)X).
\end{equation}
Fix a lift $T$ of $\tau$ modulo $4$. Then every lift of $\tau$ modulo $4$ can be written uniquely as $T+2X$ for some $X \in M_2(\Z/2\Z)$. By \eqref{eq:det_expand_2x2} we have
$$\det(I-(T+2X)) \equiv \det(I-T)+2L_1(X) \pmod 4.$$
Thus,
\begin{equation}\label{G12condition}
    (T+2X) \in G_1^2(\tau) \iff \det(I-T) \equiv 2L_1(X) \pmod 4.
\end{equation}
Also, we have
\begin{equation}\label{determinantcondition}
    \det(T+2X) \equiv \det T + 2L_2(X) \pmod 4.
\end{equation}

Since both $\adj(I-\tau)$ and $\adj(\tau)$ are non-zero and distinct, $L_1$ and $L_2$ are distinct linear functionals. Thus, one may choose $X_0 \in \ker(L_1) \setminus \ker(L_2)$. Consider a map
$$\iota : G_1^2(\tau) \to G_1^2(\tau), \quad \text{ where } \quad M \mapsto M+2X_0.$$
By \eqref{G12condition}, we see that the map is well-defined because $X_0 \in \ker L_1$, and hence $\iota$ defines an involution. Moreover, since $X_0 \not \in \ker L_2$, by \eqref{determinantcondition}, we have
$$\det(M+2X_0) \equiv \det M +  2 \pmod 4.$$
Thus, $\iota$ pairs elements of $G_1^2(\tau)$ with determinant classes $1$ and $-1$ modulo $4$, and hence
$$\#\{M \in G_1^2(\tau) : \det M \equiv 1 \pmod 4\} = \#\{M \in G_1^2(\tau) : \det M \equiv -1 \pmod 4\}.$$
Therefore, by \eqref{localcharacter}, we have
$$\sum_{M \in G_1^2(\tau)}\psi_4(M)
= \sum_{\substack{M \in G_1^2(\tau) \\ \det M\equiv 1 (4)}} (-1)
+ \sum_{\substack{M \in G_1^2(\tau) \\ \det M\equiv -1 (4)}} (+1)
= 0.$$

Now, let $\alpha > 2$. By \eqref{localcharacter}, $\psi_{2^\alpha}$ factors through reduction modulo 4, i.e., $\psi_{2^\alpha}(\widetilde{M}) = \psi_4(M)$, for any lift $\widetilde{M}$ of $M$ modulo $2^\alpha$. Moreover,
$$G_1^\alpha(\tau) = \bigsqcup_{M \in G_1^2(\tau)} G_2^\alpha(M).$$
Thus, by Lemma \ref{M2lifting}, we have
$$n_\alpha(\tau)
= \sum_{\widetilde{M}\in G_1^\alpha(\tau)} \psi_{2^\alpha}(\widetilde{M})
= \sum_{M \in G_1^2(\tau)} \sum_{\widetilde{M} \in G_2^\alpha(M)} \psi_{4}(M)
= 8^{\alpha-2} \sum_{M \in G_1^2(\tau)}\psi_{4}(M)
= 0.$$

Finally, we assume that $\Delta'_E \equiv 2 \pmod 4$. We define $L_1$ and $L_2$ as in \eqref{L1L2}. Again, consider $X_0 \in \ker(L_1) \setminus \ker(L_2)$. Fix a lift $\tau'$ of $\tau$ modulo $4$.  Fix a lift $T$ of $\tau'$ modulo $8$. Then every element in $G_2^3(\tau')$ can be uniquely expressed as $T + 4X$ for some $X \in M_2(\Z/2\Z)$. By the same argument as in the previous case, we have
$$(T+4X) \in G_2^3(\tau') \iff \det(I-T) \equiv 4 L_1(X) \pmod 8.$$
Also, we have
$$\det(T+4X) \equiv \det T + 4L_2(X) \pmod 8.$$
By the same argument as in the previous case, the map $G_2^3(\tau') \to G_2^3(\tau')$ given by $M \mapsto M+4X_0$ defines an involution and one can check that
$$\#\{M \in G_2^3(\tau') : \det M \equiv x \pmod 8\}
= \#\{M \in G_2^3(\tau') : \det M \equiv x +4\pmod 8\}$$
for any $x \in (\Z/8\Z)^\times$.

By \eqref{localcharacter}, for any $M \in G_2^3(\tau')$, we have
$$\chi_4(M+4X_0) =\chi_4(M) \quad \text{ and } \quad \chi_8(M+4X_0) = -\chi_8(M).$$
Thus,
$$n_3(\tau) = \sum_{\tau' \in G_1^2(\tau)} \sum_{M \in G_2^3(\tau')} \psi_8(M) = 0.$$
Now, let $\alpha > 3$. By \eqref{localcharacter}, $\psi_{2^\alpha}$ factors through reduction modulo $8$, i.e., $\psi_{2^\alpha}(\widetilde{M}) = \psi_8(M)$ for any lift $\widetilde{M}$ of $M$ modulo $2^\alpha$.  Moreover,
$$G_1^\alpha(\tau) = \bigsqcup_{M \in G_1^3(\tau)} G_3^\alpha(M).$$
Thus, by Lemma \ref{M2lifting}, we have
$$n_\alpha(\tau)
= \sum_{\widetilde{M} \in G_1^\alpha(\tau)} \psi_{2^\alpha}(\widetilde{M})
= \sum_{M \in G_1^3(\tau)} \sum_{\widetilde{M} \in G_3^\alpha(M)} \chi_8(M)
= 8^{\alpha-3} \sum_{M \in G_1^3(\tau)}\psi_8(M)
= 0.$$
This completes the proof.
\end{proof}
It suffices to compute $n_\alpha(I)$ for each case.
\begin{lemma}\label{G1alphaI}
    For $\alpha \geq 1$, we have
    $$|G_1^\alpha(I)| = 3\cdot 2^{3(\alpha-1)} - 2^{2\alpha-1}.$$
\end{lemma}

\begin{proof}
    The case $\alpha = 1$ is obvious. 

    Let $\alpha \geq 2$. Then any lift of $I \in \GL_2(\Z/2\Z)$ modulo $2^\alpha$ can be uniquely expressed as $I+2X$ for some $X \in M_2(\Z/2^{\alpha-1}\Z)$. Note that
        \begin{equation}\label{G1alphacondition}
        I + 2X \in G_1^\alpha(I) \iff \det(I-(I+2X)) = \det (2X) \equiv 0 \pmod {2^\alpha} \iff \det X \equiv 0 \pmod {2^{\alpha-2}}.
    \end{equation}
    The condition on the right-hand side is trivial if $\alpha = 2$. Thus, every lift of $I$ modulo $4$ lies in $G_1^2(I)$, and hence $|G_1^2(I)| = 16$ which aligns with the formula for $\alpha = 2$.

    For $\alpha \geq 3$, we use induction. By \eqref{G1alphacondition}, we obtain an one-to-one correspondence $G_1^\alpha(I) \to W_{\alpha-1}$ by $I+2X \mapsto X$. Note that $|V_1| = |M_2(\Z/2\Z)| - |\GL_2(\Z/2\Z)| = 10$. Hence,
    $$|G_1^3(I)| = |W_2| = 16|V_1| = 160,$$
    which agrees with the formula.

    If $\alpha \geq 4$, we have
    \begin{align*}
        |G_1^\alpha(I)| = |W_{\alpha-1}| = 16|V_{\alpha-2}| &= 16(|V_{\alpha-2,*}| + |V_{\alpha-2,0}|) \\
        &=16 \cdot 9 \cdot 2^{3(\alpha-2)-3} + 16|W_{\alpha-3}| = 9 \cdot 2^{3\alpha-5} + 16|G_1^{\alpha-2}(I)|.
    \end{align*}
        Solving the recursive relation by induction, we obtain the desired result.
\end{proof}

\begin{lemma}\label{Delta'E1mod4}
    Let $E/\Q$ be a Serre curve with $\Delta'_E\equiv 1 \pmod 4$. Then
    $$|Y_{2^\alpha,+}| - |Y_{2^\alpha,-}| = -2^{2\alpha-1}.$$
\end{lemma}

\begin{proof}
        If $\Delta'_E \equiv 1 \pmod 4$, then $\psi_{2^\alpha}$ factors through reduction modulo 2 by \eqref{localcharacter}, i.e.,        $\psi_{2^\alpha}(\widetilde{M}) = \epsilon(M)$ for any lift $\widetilde{M}$ of $M$ modulo $2^\alpha$. Thus, we have $n_\alpha(I)  = |G_1^\alpha(I)|$. The results follows from \eqref{Yalphasum}, Lemma \ref{nalphatau}, and Lemma \ref{G1alphaI}.
\end{proof}

\begin{lemma}\label{Delta'E3mod4}
    Let $E/\Q$ be a Serre curve with $\Delta'_E\equiv 3 \pmod 4$. Then
     $$|Y_{2^\alpha,+}| - |Y_{2^\alpha,-}| =  \begin{cases}
            -2^{2\alpha-1} & \text{ if } \alpha \geq 3,\\
            0 & \text{ if } \alpha = 2.
        \end{cases}$$
\end{lemma}

\begin{proof}
    By \eqref{conditiononalpha} and Lemma \ref{nalphatau}, it suffices to compute $n_\alpha(I)$ for $\alpha \geq 2$. Any lift of $I \in \GL_2(\Z/2\Z)$ modulo $2^\alpha$ can be uniquely expressed as $I+2X$ for some $X \in M_2(\Z/2^{\alpha-1}\Z)$. By \eqref{eq:det_expand_2x2}, we have 
    $$\det(I+2X) \equiv 1 + 2 \Tr X \pmod 4,$$
    and hence, by the definition of the character, we have
    \begin{equation}\label{mod4localcharacter}
        \psi_{2^\alpha}(I+2X) = \psi_4(I+2X \neghs \pmod 4) = \chi_4(I+2X \neghs \pmod 4) = (-1)^{\Tr X \neghs \pmod 2}.
    \end{equation}
    
    For $\alpha = 2$, as proven in Lemma \ref{G1alphaI}, $I+2X \in G_1^2(I)$ for all $X \in M_2(\Z/2\Z)$. Since there are equal numbers of matrices with trace 0 and trace 1 in $M_2(\Z/2\Z)$, we obtain:
    $$n_2(I) = \sum_{M \in G_1^2(I)} \psi_4(M) = \sum_{X \in M_2(\Z/2\Z)} (-1)^{\Tr X \neghs \pmod 2} = 8 \cdot (1) + 8 \cdot (-1) = 0.$$

    For $\alpha \geq 3$, recall the map $G_1^\alpha(I) \to W_{\alpha-1}$ given by $I+2X \mapsto X$ from the proof of Lemma \ref{G1alphaI}. Using the partition in \eqref{WalphaValphapartitioned}, we rewrite $n_\alpha(I)$ as:
    \begin{equation}\label{nalphaI}
        n_\alpha(I) = \sum_{X \in W_{\alpha-1}} (-1)^{\Tr X \neghs \pmod 2} = \sum_{X \in W_{\alpha-1,0}} (-1)^{\Tr X \neghs \pmod 2} + \sum_{X \in W_{\alpha-1,*}} (-1)^{\Tr X \neghs \pmod 2}.
    \end{equation}
    For $X \in W_{\alpha-1,0}$, we have $X \equiv 0 \pmod 2$, which implies $\Tr X \equiv 0 \pmod 2$. Thus, the first sum in \eqref{nalphaI} is equal to $|W_{\alpha-1}|$. We proved that $|G_1^\alpha(I)| = |W_{\alpha-1}|$ in the proof of Lemma \ref{G1alphaI}. Thus, by Lemma \ref{Fixageq2} and Lemma \ref{G1alphaI}, the cardinality of $W_{\alpha-1,0}$ is given by:
    $$|W_{\alpha-1,0}| = |W_{\alpha-1}| - |W_{\alpha-1,*}| = (3 \cdot 2^{3(\alpha-1)} - 2^{2\alpha-1}) - 9 \cdot 2^{3\alpha-5} = 3 \cdot 2^{3\alpha-5} - 2^{2\alpha-1}.$$
    
    For $X \in W_{\alpha-1,*}$, $X \pmod 2$ is a non-zero singular matrix. There are nine non-zero singular matrices over $\Z/2\Z$ where three of them have trace 0 and six have trace 1. By Lemma \ref{M2lifting}, each of the nine non-zero singular matrices over $\Z/2\Z$ has $8^{\alpha-2}$ lifts in $W_{\alpha-1,*}$. Thus, the second sum in \eqref{nalphaI} becomes:
    $$\sum_{X \in W_{\alpha-1,*}} (-1)^{\Tr X \neghs \pmod 2} = 3 \cdot 2^{3\alpha-5} \cdot (1) + 6 \cdot 2^{3\alpha-5} \cdot (-1) = -3 \cdot 2^{3\alpha-5}.$$
    
    Combining these two sums, we have:
    $$|Y_{2^\alpha,+}| - |Y_{2^\alpha,-}| = n_\alpha(I) = (3 \cdot 2^{3\alpha-5} - 2^{2\alpha-1}) - 3 \cdot 2^{3\alpha-5} = -2^{2\alpha-1},$$
    which completes the proof.
\end{proof}

It remains to treat the case $\Delta'_E \equiv 2 \pmod 4$. We handle the cases $3\leq \alpha\leq 4$ and $\alpha\geq 5$ separately.

\begin{lemma}\label{Delta'E2case1}
    Let $E/\Q$ be a Serre curve with $\Delta'_E\equiv 2 \pmod 4$. For $3 \leq \alpha \leq 4$,
    $$|Y_{2^\alpha,+}| - |Y_{2^\alpha,-}| =0.$$
\end{lemma}

\begin{proof}
Recall that any element in $G_1^\alpha(I)$ can be expressed as $I+2X$ for some $X \in W_{\alpha-1}$ as argued in the proof of Lemma \ref{G1alphaI}.

We first prove the case $\alpha = 3$. Fix $X_0 \in M_2(\Z/2\Z)$ with an odd trace and define an map 
$$\iota : G_1^3(I) \to G_1^3(I) \quad \text{ where } \quad \iota(M) = M+4X_0.$$
Let us verify that the map is well-defined. For any $M = I+2X \in G_1^3(I)$, by \eqref{eq:det_expand_2x2}, we have
$$\det(I-\iota(M)) = \det(-2X-4X_0) = 4\det(X+2X_0) \equiv 4 \det(X) \equiv 0 \pmod 8,$$
since $X \in W_2$. Thus, $\iota$ is a well-defined involution.

Moreover, since $M = I+2X \equiv I \pmod 2$, we have $\adj(M) \equiv I \pmod 2$. Thus, by \eqref{eq:det_expand_2x2}, 
$$\det(\iota(M)) = \det(M+4X_0) \equiv \det M + 4 \Tr(\adj(M)X_0) \equiv \det M + 4 \pmod 8.$$
Hence
$$\chi_4(\iota(M))=\chi_4(M) \qquad \text{ and } \qquad \chi_8(\iota(M))=-\chi_8(M).$$
By \eqref{localcharacter}, in both cases $\Delta'_E \equiv \pm 2 \pmod 8$, the involution $\iota$ switches the value of $\psi_8$. Thus, we have
$$\#\{M \in G_1^3(I) : \psi_8(M)=1\} = \#\{M \in G_1^3(I) : \psi_8(M) = -1\}.$$
Therefore, by \eqref{Yalphasum} and Lemma \ref{nalphatau}, we have
$$|Y_{8,+}| -|Y_{8,-}| = n_3(I) = \sum_{M \in G_1^3(I)} \psi_8(M) = 0.$$
This proves the case $\alpha = 3$.

Now, let us prove the case $\alpha=4$. Using the partition in \eqref{WalphaValphapartitioned}, we can express $n_4(I)$ as follows:
\begin{equation}\label{n4I}
    n_4(I) = \sum_{X \in W_{3,0}} \psi_{16}(I+2X) + \sum_{X \in W_{3,*}} \psi_{16}(I+2X).
\end{equation}
We first compute the first sum in \eqref{n4I}. From the definition of $W_{3,0}$, one can check that
$$M_2(\Z/4\Z) \to W_{3,0} \quad \text{ where } \quad Z\mapsto 2Z$$
is well-defined and bijective. Thus, any matrix $I+2X \in G_1^4(I)$ with $X \in W_{3,0}$ can be uniquely expressed as $I+4Z$ for some $Z \in M_2(\Z/4\Z)$. By \eqref{eq:det_expand_2x2}, we have $\det(I+4Z) \equiv 1 + 4 \Tr Z \pmod 8$. Therefore,
$$\chi_4(I+4Z) = \chi_4(I) = 1 \quad \text{ and } \quad \chi_8(I+4Z) = (-1)^{\Tr Z \neghs \pmod 2}.$$
By \eqref{localcharacter}, in both cases $\Delta'_E \equiv \pm 2 \pmod 8$, we have
$$\psi_{16}(I+4Z) = (-1)^{\Tr Z \neghs \pmod 2}.$$
Since there are equal numbers of odd and even traces in $M_2(\Z/4\Z),$ we have
$$\sum_{X \in W_{3,0}} \psi_{16}(I+2X)= 0.$$

Now, we compute the second sum in \eqref{n4I}. As shown in the proof of Lemma \ref{Fixageq2}, the natural projection defines a surjective map $\pi: W_{3,*} \to V_{2,*}$ and each fiber has size 16.

Fix $A\in V_{2,*}$, and choose a lift $\widetilde A$ of $A$ modulo $8$. Then the fiber over $A$ is given by
$$\pi^{-1}(A)=\{\widetilde A+4Z: Z\in M_2(\Z/2\Z)\}.$$
Thus, for any $X = \widetilde{A}+4Z \in \pi^{-1}(A)$, we have
$$I+2X = I+2(\widetilde{A}+4Z) \equiv I+2\widetilde{A} \pmod 8.$$
Since $\psi_{16}$ factors through reduction modulo $8$, $\psi_{16}(I+2(\widetilde{A}+4Z))$ is independent of the choice of $Z$. In particular, $\psi_{16}(I+2X)$ is constant on each fiber of $\pi$. Therefore, we may define
$$\theta:V_{2,*}\to \{\pm1\} \qquad \text{where} \qquad \theta(A)\coloneqq \psi_{16}(I+2\widetilde A),$$
where $\widetilde{A}$ is any lift of $A$ modulo $8$. As we argued earlier, the map is well-defined. We have
$$\sum_{X\in W_{3,*}}\psi_{16}(I+2X)=16\sum_{A\in V_{2,*}}\theta(A).$$
Since $A \in V_{2,*}$, and hence $\det A\equiv0\pmod4$, by \eqref{eq:det_expand_2x2}, we have
$$\det(I+2\widetilde A) \equiv1+2\Tr A\pmod 8.$$
Thus, $\theta(A)$ depends only on $\Tr A \pmod 4$.
If $\Delta'_E \equiv 2 \pmod 8$, we have $\psi_{16} = \epsilon \chi_8$, and hence
$$\theta(A)=\begin{cases}
+1 & \text{if } \Tr A\equiv0,3\pmod4,\\
-1 & \text{if } \Tr A\equiv1,2\pmod4.
\end{cases}$$
If $\Delta'_E\equiv6\pmod8$, we have $\psi_{16}=\epsilon\chi_4\chi_8$, and hence
$$\theta(A) = \begin{cases}
+1 & \text{if } \Tr A\equiv0,1\pmod4,\\
-1 & \text{if } \Tr A\equiv2,3\pmod4.
\end{cases}$$
Either case, by Lemma \ref{V2startracecount}, we have
$$\sum_{A \in V_{2,*}} \theta(A) = 0.$$
This completes the proof.
\end{proof}

\begin{lemma}\label{V2startracecount}
Let
$$V_{2,*}=\{A\in M_2(\Z/4\Z): A\not\equiv 0 \pmod 2,\ \det A\equiv 0 \pmod 4\}.$$
Then
$$\#\{A \in V_{2,*} : \Tr A \equiv t \pmod 4\} = \begin{cases}
    12 & \text{ if } t\in \{0,2\}, \\
    24 & \text{ if } t \in \{1,3\}.
\end{cases}$$
\end{lemma}

\begin{proof}
Let $A_0 \in M_2(\Z/2\Z)$ be a nonzero singular matrix, and let $A$ be a lift of $A_0$ modulo $4$. We write any other lifts of $A_0$ as $A+2X$ for some $X \in M_2(\Z/2\Z)$. Define linear functionals $L$ and $T$ on $M_2(\Z/2\Z)$ as follows:
$$L(X)=\Tr(\adj(A_0)X)\qquad \text{ and } \qquad T(X)=\Tr X$$
Since $A_0$ has rank $1$, we have $\adj(A_0)\neq 0$, and hence $L$ is nonzero. By \eqref{eq:det_expand_2x2},
$$\det(A+2X)\equiv \det A+2L(X)\pmod 4.$$
By Lemma \ref{M2lifting}, exactly eight choices of $X$ satisfy $A+2X \in V_{2,*}$.

Since $A_0$ is a nonzero singular matrix, $\adj(A_0) \neq I$, and hence $L$ and $T$ are distinct. Therefore, on each fiber of $L$, the functional $T$ takes the values $0$ and $1$ equally often. It follows that these eight lifts split evenly between
the two trace classes $\Tr A$ and $\Tr A + 2$ modulo $4$.

Finally, among the nonzero singular matrices in $M_2(\Z/2\Z)$, exactly $3$ have trace $0$ and exactly $6$ have trace $1$. Therefore the trace classes $0$ and $2$ modulo $4$ each receive $3\cdot 4=12$ lifts, while the trace classes $1$ and $3$ modulo $4$ each receive $6\cdot 4=24$ lifts. This proves the claim.
\end{proof}

\begin{lemma}\label{Delta'E2case2}
    Let $E/\Q$ be a Serre curve with $\Delta'_E\equiv 2 \pmod 4$. If $\alpha \geq 5$, then
$$|Y_{2^\alpha,+}| - |Y_{2^\alpha,-}| = -2^{2\alpha-1}.$$
\end{lemma}

\begin{proof}
Recall that any element in $G_1^\alpha(I)$ can be expressed as $I+2X$ for some $X \in W_{\alpha-1}$ as argued in the proof of Lemma \ref{G1alphaI}. Since $\alpha\geq 5$, for every $X\in W_{\alpha-1}$, we have $\det X \equiv 0 \pmod 8$, since $\alpha -2 \geq 3$. By \eqref{eq:det_expand_2x2}, we have
$$\det(I+2X) = 1+2\Tr X+4\det X \equiv 1+2\Tr X \pmod 8.$$
Thus, as argued in the proof of Lemma \ref{Delta'E2case1}, the value $\psi_{2^\alpha}(I+2X)$ depends only on $\Tr X \pmod 4$. Define
$$S_\alpha \coloneqq \sum_{X\in W_{\alpha-1}}\psi_{2^\alpha}(I+2X).$$
It suffices to prove that $S_\alpha = -2^{2\alpha-1}$.

As shown in the proof of Lemma \ref{Fixageq2}, the natural projection defines a surjective map $W_{\alpha-1} \to V_{\alpha-2}$ of fiber size $16$. Note that if $X\in W_{\alpha-1}$ maps to $A\in V_{\alpha-2}$, then $I+2X$ reduces to $I+2A$ modulo $2^{\alpha-1}$, and hence
$$\psi_{2^\alpha}(I+2X)=\psi_{2^{\alpha-1}}(I+2A),$$
since both characters factor through the same reduction modulo $8$. Using \eqref{WalphaValphapartitioned}, we have
\begin{equation}\label{Salphasum}
    S_\alpha = 16\sum_{A\in V_{\alpha-2}}\psi_{2^{\alpha-1}}(I+2A) = 16\sum_{A\in V_{\alpha-2,*}}\psi_{2^{\alpha-1}}(I+2A) + 16\sum_{A\in V_{\alpha-2,0}}\psi_{2^{\alpha-1}}(I+2A).
\end{equation}

Let us compute the first sum in \eqref{Salphasum}. The reduction modulo $4$ maps $V_{\alpha-2,*}$ onto $V_{2,*}$, and each fiber has size $8^{\alpha-4}$ by Lemma \ref{M2lifting}. Since $\psi_{2^{\alpha-1}}(I+2A)$ depends only on $\Tr A \pmod 4$, Lemma \ref{V2startracecount} gives
$$\sum_{A\in V_{\alpha-2,*}}\psi_{2^{\alpha-1}}(I+2A)=0.$$

Now we compute the second sum in \eqref{Salphasum}. Let $A \in V_{\alpha-2,0}$. Note that $A$ can be expressed uniquely as $2B$ for some $B \in W_{\alpha-3}$. By \eqref{eq:det_expand_2x2}, we have
$$\det(I+2A) = \det(I+4B) \equiv 1 + 4 \Tr B \pmod 8.$$
Therefore, in both cases $\Delta'_E\equiv \pm 2 \pmod8$, we have 
$$\psi_{2^{\alpha-1}}(I+4B)=(-1)^{\Tr B \neghs \pmod 2},$$
and hence
$$\sum_{A\in V_{\alpha-2,0}}\psi_{2^{\alpha-1}}(I+2A)=\sum_{B\in W_{\alpha-3}}(-1)^{\Tr B \neghs \pmod 2}.$$
Using the same argument and computation as in the proof of Lemma \ref{Delta'E3mod4}, we have
$$\sum_{B\in W_{\alpha-3}}(-1)^{\Tr B \neghs \pmod 2}=-2^{2(\alpha-2)-1} \implies S_\alpha = 16 \cdot -2^{2(\alpha-2)-1} = -2^{2\alpha-1}.$$
This completes the proof.
\end{proof}
By Lemmas \ref{Delta'E1mod4}, \ref{Delta'E3mod4}, \ref{Delta'E2case1}, and \ref{Delta'E2case2}, we may conclude as follows.
\begin{proposition}\label{2adiccase}
    Let $E/\Q$ be a Serre curve of adelic level $m_E$. Let $\Delta'_E$ denote the squarefree part of the discriminant of $E$. Let $\alpha \geq v_2(m_E)$. Then
    \begin{itemize}
        \item If $\Delta'_E\equiv 1 \pmod 4$, then 
        $$|Y_{2^\alpha,+}| - |Y_{2^\alpha,-}| = -2^{2\alpha-1}.$$
        \item If $\Delta'_E \equiv 3 \pmod 4$, then
        $$|Y_{2^\alpha,+}| - |Y_{2^\alpha,-}| = \begin{cases}
            0 & \text{ if } \alpha =2 \\
            -2^{2\alpha-1} & \text{ if } \alpha \geq 3.
        \end{cases}$$
        \item If $\Delta'_E \equiv 2 \pmod 4$, then
        $$|Y_{2^\alpha,+}| - |Y_{2^\alpha,-}| = \begin{cases}
            0 & \text{ if } 3 \leq \alpha \leq 4 \\
            -2^{2\alpha-1} & \text{ if } \alpha \geq 5.
        \end{cases}$$
    \end{itemize}
\end{proposition}

\subsection{Odd Prime powers}
Now, we change our attention $Y_{\ell^\alpha,\pm}$ for odd prime powers $\ell^\alpha$. Recall that
$$Y_{\ell^\alpha, \pm} = \left\{A \in \GL_2(\Z/\ell^\alpha \Z) : \left(\frac{\det A_\ell}{\ell} \right) = \pm 1 \text{ and } \det(I-A) \equiv 0 \pmod {\ell^\alpha}\right\},$$
where $A_\ell$ denotes $A \pmod \ell$. Define 
$$N_\alpha(t,d) \coloneqq \#\{A \in \GL_2(\Z/\ell^\alpha \Z ) : \Tr A = t, \det A = d\}.$$
Since
$$\det(I-A) \equiv 1 + \det A - \Tr A \pmod {\ell^\alpha},$$
by \eqref{eq:det_expand_2x2}, we aim to count $N_\alpha(d+1,d)$ for various $d \in (\Z/\ell^\alpha \Z)^\times$ with $\left(\frac{d \pmod \ell}{\ell}\right) = \pm 1$.

Let $x \in \Z/\ell^\alpha \Z$. We define
$$v_\ell(x) = \max\{0\leq v\leq \alpha : x\equiv 0 \pmod{\ell^v}\}.$$
In particular, $v_\ell(0) = \alpha$.
\begin{lemma}\label{Mxsize}
Let $x \in \Z/\ell^\alpha \Z$ and set $v = v_\ell(x)$. Define
$$M(x) \coloneqq \#\{(b,c) \in (\Z/\ell^\alpha \Z)^2 : bc \equiv x \pmod {\ell^\alpha}\}.$$
Then
$$M(x) = \begin{cases}
(v+1)\varphi(\ell^{\alpha}) & \text{if } x \not\equiv 0 \pmod{\ell^\alpha}, \\
\alpha\varphi(\ell^{\alpha}) + \ell^{\alpha} & \text{if } x \equiv 0 \pmod{\ell^\alpha}.\end{cases}$$
\end{lemma}

\begin{proof}
Let $i = v_\ell(b)$. If $b \neq 0$, then $b$ can be uniquely expressed as $\ell^i u$ for some $u \in (\Z/\ell^{\alpha-i}\Z)^\times$. We count the number of solutions per each $i$. 

Suppose $x \neq 0$. If $i > v$, then the congruence equation has no solution. If $i \leq v$, we get
$$\ell^i u c \equiv x \pmod {\ell^\alpha} \iff c \equiv u^{-1}x/\ell^i \pmod {\ell^{\alpha-i}},$$
where $x/\ell^i$ denotes the well-defined class in $\Z/\ell^{\alpha-i}\Z$ obtained by dividing $x$ by $\ell^i$. For each $c$ satisfying the above congruence equation, there are exactly $\ell^i$ lifts of them in $\Z/\ell^\alpha \Z$. Therefore, the number of solutions is $\varphi(\ell^{\alpha-i}) \cdot \ell^i = \varphi(\ell^\alpha)$. Summing over $0 \leq i \leq v$, we obtain the first case.

Suppose $x = 0$. Suppose $b \neq 0$. Then we have
$$bc \equiv 0 \pmod {\ell^\alpha} \iff c \equiv 0 \pmod {\ell^{\alpha-i}}.$$
Thus, using the same argument as above, we see that there are $\alpha \varphi(\ell^\alpha)$ solutions in this case. Finally, if $b = 0$, then any choice of $c$ satisfies the congruence equation, contributing another $\ell^\alpha$ solutions. This completes the proof.
\end{proof}

\begin{lemma}\label{Nalphatd}
    Let $\ell^\alpha$ be an odd prime power. Fix $d \in (\Z/\ell^\alpha \Z)^\times$ and $t \in \Z/\ell^\alpha \Z$. For each $1 \leq k \leq \alpha$, define
    $$\Delta \coloneqq t^2-4d \quad \text{ and } \quad S_k \coloneqq \#\{y \in \Z/\ell^\alpha \Z : y^2 \equiv \Delta \pmod {\ell^k}\}.$$
    Then we have
    $$N_\alpha(t,d) = \varphi(\ell^\alpha) \left(\ell^\alpha + \sum_{k=1}^{\alpha-1} S_k\right) + \ell^\alpha S_\alpha.$$
\end{lemma}
\begin{proof}
    Let $A \in \GL_2(\Z/\ell^\alpha \Z)$ with $\Tr A = t$, i.e.,
    $$A = \begin{pmatrix}
        a & b \\ c & t-a
    \end{pmatrix}.$$
    Then $\det A =d$ is equivalent to that $a(t-a) -bc \equiv d \pmod {\ell^\alpha}$. Set $x_a \coloneqq a(t-a)-d$. Using the notation from Lemma \ref{Mxsize}, we obtain
    \begin{equation}\label{originalNalphatd}
    N_\alpha(t,d) = \sum_{a \in \Z/\ell^\alpha \Z} M(x_a).
    \end{equation}
    Note that
    $$4x_a = 4a(t-a)-4d = \Delta - (2a-t)^2.$$
    Since $\ell$ is an odd prime, $v_\ell(x_a) = v_\ell(4x_a)$. Since $M(x)$ only depends on $v_\ell(x)$, by Lemma \ref{Mxsize}, we have
    \begin{equation}\label{linearchange}
        M(x_a) = M(4x_a) = M(\Delta - (2a-t)^2).
    \end{equation}
    Let $y =2a-t$. The map $a \mapsto 2a-t$ gives a bijection on $\Z/\ell^\alpha \Z$. By \eqref{linearchange}, the summation \eqref{originalNalphatd} can be written as
    $$N_\alpha(t,d) = \sum_{y \in \Z/\ell^\alpha \Z} M(\Delta-y^2).$$
    Thus, we are left to find how $v_\ell(\Delta-y^2)$ is distributed as $y$ varies.

    Note that the sets counted by $S_k$ are nested. Thus, we have
    $$C_k \coloneqq\#\{y \in \Z/\ell^\alpha \Z : v_\ell(\Delta-y^2) = k\} = \begin{cases}
        \ell^\alpha - S_1 & \text{ if } k =0,
    \\
    S_k - S_{k+1} & \text{ if } 1 \leq k \leq \alpha-1, \\
    S_\alpha & \text{ if } k = \alpha.
    \end{cases}$$
    Applying Lemma \ref{Mxsize}, we obtain
    \begin{align*}
        N_\alpha(t,d) &= \sum_{y \in \Z/\ell^\alpha \Z}M(\Delta-y^2) = \sum_{k =0}^{\alpha-1} C_k  (k+1) \varphi(\ell^\alpha) + C_\alpha  (\alpha\varphi(\ell^\alpha) + \ell^\alpha) \\
        &= \varphi(\ell^\alpha)(\ell^\alpha -S_1) + \varphi(\ell^\alpha) \sum_{k=1}^{\alpha-1} (S_k-S_{k+1})(k+1) + (\alpha \varphi(\ell^\alpha) + \ell^\alpha)(S_\alpha) \\
        &= \varphi(\ell^\alpha) \left(\ell^\alpha + \sum_{k=1}^{\alpha-1} S_k\right) + \ell^\alpha S_\alpha.
    \end{align*}
    This completes the proof.
\end{proof}

\begin{lemma}\label{Skvalue}
    Let $\Delta \in \Z/\ell^\alpha \Z$ and $s = v_\ell(\Delta)$. If $s < \alpha$, write $\Delta = \ell^s u$ for some unit $u$ modulo $\ell$. With the notation as in Lemma \ref{Nalphatd}, we have
    $$S_k = \begin{cases}
        \ell^{\alpha - \lceil k/2\rceil} &\text{ if } k \leq s \\
        2\ell^{\alpha-k+s/2} & \text{ if } k > s, s \text{ is even, and }\left(\frac{u}{\ell}\right) = 1,\\
        0 &\text{ otherwise.}
    \end{cases}$$
\end{lemma}
\begin{proof}
    For $1 \leq k \leq \alpha$, define
    $$s_k \coloneqq \#\{y \in \Z/\ell^k \Z : y^2 \equiv \Delta \pmod {\ell^k}\}.$$
    Since each solution modulo $\ell^k$ has exactly $\ell^{\alpha-k}$ lifts to a solution modulo $\ell^\alpha$, we have $S_k = s_k\ell^{\alpha-k}$.

    Suppose $k \leq s$. Then $\Delta \equiv 0 \pmod {\ell^k}$. Note that
   $$y^2 \equiv 0 \pmod {\ell^k} \iff y \equiv 0 \pmod {\ell^{\lceil k/2\rceil}},$$
   and there are exactly $\ell^{k- \lceil k/2 \rceil}$ solutions modulo $\ell^k$ satisfying the right-hand side. Thus, $S_k = \ell^{k - \lceil k/2 \rceil} \cdot \ell^{\alpha-k} = \ell^{\alpha - \lceil k/2\rceil}$.

   Now, suppose $k > s$. Note that $\Delta' \neq 0$ this case, and hence it can be expressed as $\ell^s u$ for some $0 \leq s < \alpha$ and a unit $u$ modulo $\ell$. Obviously, the congruence equation $y^2 \equiv \Delta \pmod {\ell^k}$ admits no solutions if $s$ is odd. Hence, from now on, assume that $s$ is even and write $s = 2h$. Since $k > 2h$, any solution must be of the form $y = \ell^h w$ for some unit $w$ modulo $\ell$. Thus we observe that
   $$y^2 \equiv \Delta \pmod {\ell^k} \iff w^2 \equiv u \pmod {\ell^{k-2h}}.$$
   The congruence equation $y^2 \equiv \Delta \pmod {\ell}$ admits no solutions if $\left(\frac{u}{\ell}\right) = -1$ and has two solutions if $\left(\frac{u}{\ell}\right) = 1$. In the latter case, since $2w \not \equiv 0 \pmod {\ell}$, Hensel's lemma implies that each solution modulo $\ell$ lifts uniquely to a solution modulo $\ell^{k-2h}$.
   
     Let $w_0 \in \Z/\ell^{k-2h} \Z$ be a such solution. For each $t \in \Z/\ell^h \Z$, define
    $$w_t \coloneqq w_0 + \ell^{k-2h} t \in \Z/\ell^{k-h}\Z \quad \text{ and } \quad y_t \coloneqq \ell^h w_t \in \Z/\ell^k\Z.$$
    Then we have $y_t^2 \equiv \Delta \pmod{\ell^k}$. If $t \not \equiv t' \pmod {\ell^h}$, we see that $y_t-y_{t'} = \ell^{k-h} (t-t') \not \equiv 0 \pmod {\ell^k}$, and hence each $y_t$ is distinct. Therefore, each solution $w_0$ gives rise to exactly $\ell^h$ solutions $y \pmod {\ell^k}$, and hence $s_k = 2\ell^h$. This completes the proof.
\end{proof}
We now compute $N_{\alpha}(d+1,d)$ for various $d \in (\Z/\ell^\alpha \Z)^\times$.
\begin{lemma}\label{Nalphad+1d}
    Let $r = v_\ell(d-1)$. Then we have
    $$N_\alpha(d+1,d) = \begin{cases}
        (\ell+1)\ell^{2\alpha-1} - \ell^{\lfloor \frac{3\alpha-1}{2}\rfloor} & \text{ if } r \geq \alpha/2, \\
        (\ell+1) \ell^{2\alpha-1} & \text{ if } r < \alpha/2.
    \end{cases}$$
\end{lemma}
\begin{proof}
    First of all, note that $\Delta = (d+1)^2-4d = (d-1)^2 \in (\Z/\ell^\alpha\Z)$. Thus, $\Delta$ is either 0 or it can be expressed as $\ell^{2r}u$ for some quadratic residue $u$ modulo $\ell$.

    We first consider the case $2r \geq \alpha$. Then $S_k = \ell^{\alpha - \lceil k/2\rceil}$ for each $1 \leq k \leq \alpha$ by Lemma \ref{Skvalue}. We partition the sum into two terms with $k$ even and those with $k$ odd.

    Suppose $\alpha$ is even and write $\alpha = 2m$. Then 
    \begin{align*}
        \sum_{k=1}^{2m-1} S_k = \sum_{j=1}^m S_{2j-1} + \sum_{j=1}^{m-1} S_{2j} = \sum_{j=1}^m \ell^{2m -j} + \sum_{j=1}^{m-1}\ell^{2m-j} = \frac{2(\ell^{2m} - \ell^{m+1})}{\ell-1} + \ell^m.
    \end{align*}
    Thus, by Lemma \ref{Nalphatd},
    \begin{align*}
        N_\alpha(d+1,d) &= \varphi(\ell^{2m}) \left(\ell^{2m} + \frac{2(\ell^{2m} - \ell^{m+1})}{\ell-1} + \ell^m\right) + \ell^{2m} \cdot \ell^m \\
        &=(\ell-1)\ell^{4m-1} + 2\ell^{3m}(\ell^{m-1}-1)+(\ell-1)\ell^{3m-1} + \ell^{3m} = (\ell+1)\ell^{4m-1} - \ell^{3m-1}.
    \end{align*}
    Note that $3m-1 = \lfloor \frac{3\alpha-1}{2}\rfloor$ and hence the desired result is obtained.

    On the other hand, suppose $\alpha$ is odd and write $\alpha = 2m+1$. Then
    \begin{equation}\label{evencase}
            \sum_{k=1}^{2m} S_k = \sum_{j=1}^m S_{2j-1} + \sum_{j=1}^{m} S_{2j} = 2\sum_{j=1}^m \ell^{2m+1-j} = 2 \left(\frac{\ell^{2m+1} - \ell^{m+1}}{\ell-1} \right).
    \end{equation}
    Again, by Lemma \ref{Nalphatd},
    \begin{align*}
        N_\alpha(d+1,d) &= \varphi(\ell^{2m+1}) \left(\ell^{2m+1} + 2 \left(\frac{\ell^{2m+1} - \ell^{m+1}}{\ell-1}\right)\right) + \ell^{2m+1} \cdot \ell^m \\
        &=(\ell-1)\ell^{4m+1} + 2\ell^{3m+1}(\ell^{m}-1)+ \ell^{3m+1} = (\ell+1)\ell^{4m+1} - \ell^{3m+1}.
    \end{align*}
    Note that $3m+1 = \lfloor \frac{3\alpha-1}{2}\rfloor$ and hence the desired result is obtained.

    Now, we turn our attention to the case $2r < \alpha$. By Lemma \ref{Skvalue}, we have
    $$S_k = \begin{cases}
        \ell^{\alpha-\lceil k/2\rceil} & \text{ if } 1 \leq k \leq 2r, \\
        2\ell^{\alpha-k+r} & \text{ if } 2r < k \leq \alpha. 
    \end{cases}$$
    Thus, we have 
    \begin{equation}
                \sum_{k=1}^{\alpha-1} S_k = \sum_{k=1}^{2r} \ell^{\alpha - \lceil k/2\rceil} + \sum_{k=2r+1} ^{\alpha-1} 2\ell^{\alpha-k+r}. 
    \end{equation}
The first sum can be computed by separating the terms according to the parity of $k$. That is,
    $$\sum_{k=1}^{2r} \ell^{\alpha - \lceil k/2\rceil} = 2\sum_{j=1}^r \ell^{\alpha-j} = 2 \left( \frac{\ell^\alpha - \ell^{\alpha-r}}{\ell-1}\right).$$
The second sum can be simplified using the geometric series formula:
    $$\sum_{k=2r+1}^{\alpha-1} 2\ell^{\alpha-k+r} = 2\ell^{\alpha+r} \sum_{k=2r+1}^{\alpha-1} \ell^{-k} = 2\ell^{\alpha+r} \left( \frac{\ell^{-2r} - \ell^{1-\alpha}}{\ell-1}\right) = 2 \left(\frac{\ell^{\alpha-r} - \ell^{1+r}}{\ell-1}\right).$$
    Therefore, we obtain
    $$\sum_{k=1}^{\alpha-1}S_k =  2 \left( \frac{\ell^\alpha - \ell^{\alpha-r}}{\ell-1}\right) +2 \left(\frac{\ell^{\alpha-r} - \ell^{1+r}}{\ell-1}\right) = 2 \left(\frac{\ell^{\alpha}-\ell^{1+r}}{\ell-1}\right).$$
    Since $2r < \alpha$, we have $S_\alpha = 2\ell^r$. Therefore, by Lemma \ref{Nalphatd}, we obtain
    \begin{align*}
        N_{\alpha}(d+1,d) &= \varphi(\ell^{\alpha}) \left( \ell^\alpha + 2 \left(\frac{\ell^{\alpha}- \ell^{r+1}}{\ell-1}\right)\right) + \ell^{\alpha} \cdot 2\ell^r \\
        &= (\ell-1)\ell^{2\alpha-1} + 2\ell^{2\alpha-1} -2\ell^{\alpha+r} + 2\ell^{\alpha+r} = (\ell+1)\ell^{2\alpha-1}.
    \end{align*}
    This completes the proof.
\end{proof}

\begin{proposition}\label{oddadic}
    For any odd prime power $\ell^\alpha$,
    $$|Y_{\ell^\alpha,+}| - |Y_{\ell^\alpha,-}| = -\ell^{2\alpha - 1}.$$
\end{proposition}
\begin{proof}
    Recall that
    $$Y_{\ell^\alpha,\pm} = \left\{M \in \GL_2(\Z/\ell^\alpha\Z) : \det M + 1 \equiv \Tr M \pmod {\ell^\alpha} \text{ and } \left(\frac{\det M_\ell}{\ell}\right) = \pm 1\right\},$$
    and hence 
    $$|Y_{\ell^\alpha,\pm}| = \sum_{\substack{d \in (\Z/\ell^\alpha \Z)^\times \\ \left(\frac{d }{\ell}\right) = \pm 1}} N_\alpha(d+1,d).$$
    There are $\varphi(\ell^\alpha)/2$ quadratic non-residues modulo $\ell$ in $(\Z/\ell^\alpha \Z)^\times$. For any such element $d$, we have $d \not \equiv 1 \pmod \ell$, and hence $v_\ell(d-1) =0<\alpha/2$.  By Lemma \ref{Nalphad+1d},
    $$|Y_{\ell^\alpha,-}| = \sum_{\substack{d \in (\Z/\ell^\alpha \Z)^\times \\ \left(\frac{d }{\ell}\right) = -1}} N_{\alpha}(d+1,d) = \frac{\varphi(\ell^\alpha)}{2}(\ell+1)\ell^{2\alpha-1}.$$
    Likewise, there are $\varphi(\ell^\alpha)/2$ quadratic residues modulo $\ell$ in $(\Z/\ell^\alpha \Z)^\times$. Let $c = \lceil \alpha/2\rceil$. Exactly $\ell^{\alpha-c}$  elements $d \in (\Z/\ell^\alpha \Z)^\times$ satisfy $v_\ell(d-1) \geq \alpha/2$. By Lemma \ref{Nalphad+1d},
    \begin{align*}
        |Y_{\ell^\alpha,+}| &= \left( \frac{\varphi(\ell^\alpha)}{2} - \ell^{\alpha-c} \right) (\ell+1)\ell^{2\alpha-1} + \ell^{\alpha-c} \left( (\ell+1)\ell^{2\alpha-1} - \ell^{\lfloor \frac{3\alpha-1}{2}\rfloor}\right) \\
        &=\frac{\varphi(\ell^\alpha)}{2}(\ell+1)\ell^{2\alpha-1} - \ell^{\alpha-c + \lfloor \frac{3\alpha-1}{2}\rfloor},
    \end{align*}
    and hence
    $$|Y_{\ell^\alpha,+}| - |Y_{\ell^\alpha,-}| = -\ell^{\alpha-c+\lfloor \frac{3\alpha-1}{2}\rfloor}.$$
    Checking each parity of $\alpha$ separately, we see that $\alpha-c+\lfloor \frac{3\alpha-1}{2}\rfloor = 2\alpha-1$. 
\end{proof}
\begin{proof}[Proof of Theorem \ref{maintheorem}]
    Let $E/\Q$ be a Serre curve of adelic level $m_E$ and $m$ be a positive integer. Let $\Delta'_E$ denote the squarefree part of the discriminant of $E$.

    The case that $m_E \nmid m$ follows from Proposition \ref{Serremodmimage} and Proposition \ref{fullgroupcaseofconstant}.

    Suppose $m_E \mid m$. Write $m = m_1m_2$ so that $m_1 = \gcd(m,m_E^\infty)$ and $\gcd(m_2,m_E) = 1$. Suppose $\Delta'_E \equiv 3 \pmod 4$ and $v_2(m_1) = 2$ or $\Delta'_E \equiv 2 \pmod 4$ and $v_2(m_1) \leq 4$. By Proposition \ref{countingprop} and Proposition \ref{2adiccase},
    $$|\psi_{m_1}^{-1}(+1) \cap \Psi(m_1)| = \frac{1}{2}|\Psi(m_1)|.$$
    By Proposition \ref{Serremodmimage} and \eqref{correctdecomposition},
    $$C_E^{m\divi} = \frac{\frac{1}{2}|\Psi(m_1)|}{|H_E(m_1)|} \cdot \frac{|\Psi(m_2)|}{|\GL_2(\Z/m_2\Z)|} = \frac{|\Psi(m)|}{|\GL_2(\Z/m\Z)|}.$$
    In particular, the $m$-divisibility density $C_E^{m\divi}$ coincides with the average density $C^{m\divi}$.

    Finally, suppose the rest of the cases. By Propositions \ref{countingprop}, \ref{2adiccase}, and \ref{oddadic}, we have
    \begin{align*}
        |\psi_{m_1}^{-1}(+1) \cap \Psi(m_1)| = \frac{1}{2}\left( |\Psi(m_1)| + \prod_{\ell^\alpha \parallel m_1} -\ell^{2\alpha-1}\right).
    \end{align*}
    By \eqref{correctdecomposition},
    $$C_E^{m\divi} = \left( \frac{\frac{1}{2}\left( |\Psi(m_1)| + \prod_{\ell^\alpha \parallel m_1} -\ell^{2\alpha-1}\right)}{\frac{1}{2}|\GL_2(\Z/m_1\Z)|}\right) \cdot  C^{m_2\divi} = \left(C^{m_1\divi} + \prod_{\ell^\alpha \parallel m_1} \frac{-\ell^{2\alpha-1}}{|\GL_2(\Z/\ell^\alpha \Z)|}\right) \cdot C^{m_2\divi}.$$
    This completes the proof.
    \end{proof}
    Finally, we conclude this chapter with the proof of Corollary \ref{maincorollary}, that is $C_E^{m\divi} > 1/m$ for any Serre curve $E/\Q$ and $m \geq 2$.

    \begin{proof}[Proof of Corollary \ref{maincorollary}]
For every prime power $\ell^\alpha$, one can easily check that
$$C^{\ell^\alpha\divi} = \frac{1}{\varphi(\ell^\alpha)}
\left(1-\frac{1}{\varphi(\ell^\alpha)(\ell+1)}\right)=\frac{1}{\ell^\alpha}\cdot\frac{\ell}{\ell-1}\left(1-\frac{1}{\ell^{\alpha-1}(\ell^2-1)}\right)> \frac{1}{\ell^\alpha}.$$
Thus, if $m_E \nmid m$, by Proposition \ref{fullgroupcaseofconstant},
$$C^{m\divi}_E = C^{m\divi} = \prod_{\ell^\alpha \parallel m} C^{\ell^\alpha\divi} > \frac{1}{m}.$$
Now, let $m = m_1m_2$ with $m_1 = \gcd(m,m_E^\infty)$ and $\gcd(m_2,m_E) = 1$. By Theorem \ref{maintheorem}, it suffices to check that
$$C^{m_1\divi} - \prod_{\ell^\alpha \parallel m_1} \frac{\ell^{2\alpha-1}}{|\GL_2(\Z/\ell^\alpha\Z)|} > \frac{1}{m_1}.$$
   Note that
    \begin{align*}
        C^{m_1\divi} - \prod_{\ell^\alpha \parallel m_1} \frac{\ell^{2\alpha-1}}{|\GL_2(\Z/\ell^\alpha\Z)|} &= \prod_{\ell^\alpha \parallel m_1} \frac{1}{\varphi(\ell^\alpha)}\left( 1 - \frac{1}{\varphi(\ell^\alpha)(\ell+1)}\right) - \prod_{\ell^\alpha \parallel m_1} \frac{\ell^{2\alpha-1}}{\ell^{4\alpha-4}(\ell^2-1)(\ell^2-\ell)} \\
        &= \frac{1}{m_1}\prod_{\ell^\alpha \parallel m_1}\frac{\ell}{\ell-1}\left(1 - \frac{1}{\ell^{\alpha-1}(\ell^2-1)}\right) - \frac{1}{m_1} \prod_{\ell^\alpha \parallel m_1} \frac{1}{\ell^{\alpha-2}(\ell^2-1)(\ell+1)}.
    \end{align*}
    Define multiplicative functions $F$ and $H$ that are defined at prime powers as follows:
  $$F(\ell^\alpha) \coloneqq \frac{1}{\ell^{\alpha-2}(\ell^2-1)(\ell+1)} \quad \text{ and } \quad   H(\ell^\alpha) \coloneqq \frac{\ell}{\ell-1}\left( 1 - \frac{1}{\ell^{\alpha-1}(\ell^2-1)}\right).$$
    Hence, it suffices to prove that
    $$H(m_1) - F(m_1) > 1.$$
    If $\ell$ is odd or $\ell = 2$ and $\alpha \geq 2$, one can easily check that
    $$H(\ell^\alpha) - F(\ell^\alpha) - 1 = \frac{\ell^{\alpha-1}(\ell^2-1)-2\ell}{\ell^{\alpha-1}(\ell-1)(\ell^2-1)} > 0.$$
    Thus, if $v_2(m_1) \geq 2$, we have
    $$H(m_1) = \prod_{\ell^\alpha \parallel m_1} H(\ell^\alpha) > \prod_{\ell^\alpha \parallel m_1}(1+F(\ell^\alpha)) \geq 1 + \prod_{\ell^\alpha \parallel m_1} F(\ell^\alpha) = 1 + F(m_1),$$
    since each $F(\ell^\alpha)$ is positive.

    Finally, suppose $v_2(m_E) = v_2(m_1) = 1$. Write $m_1 = 2n$ for some odd integer $n$. Since $m_E \mid m$, by Lemma \ref{atleast6}, $n \geq 3$. Since $H(n) > F(n) + 1$, we have
    $$H(m_1) - F(m_1) = \frac{4}{3}H(n) - \frac{2}{3}F(n) > \frac{2}{3}F(n) + \frac{4}{3} > 1.$$
    This completes the proof.
    \end{proof}

\section{Proof of Theorem \ref{maintheorem2}}\label{maintheorem2proof}
In this section we prove Theorem \ref{maintheorem2}. Recall that $\mathcal{F} \coloneqq \mathcal{F}(A,B)$ is the family of all elliptic curves $Y^2 = X^3 + aX+b$ with $a,b\in \Z$, $|a| \leq A$, and $|b| \leq B$. We first partition $\mathcal{F}$ into the subfamily of Serre curves $\mathcal{F}^{\text{Serre}}$ and subfamily of non-Serre curves $\mathcal{F}^{\text{non-Serre}}$. Fix positive integers $k$ and $m\geq 2$. We further partition $\mathcal{F}^{\text{Serre}}$ as follows:
\begin{align*}
    \mathcal{F}_1\coloneqq \{E \in \mathcal{F}^{\text{Serre}}: C_E^{m\divi} = C^{m\divi}\} \quad \text{ and } \quad \mathcal{F}_2 \coloneqq \{E \in \mathcal{F}^\text{Serre}: C_E^{m\divi} \neq C^{m\divi}\}.
\end{align*}
We aim to bound
$$\frac{1}{|\mathcal{F}|}\left( \sum_{\substack{E \in \mathcal{F}_1}}|C_E^{m\divi}-C^{m\divi}|^k +  \sum_{\substack{E \in \mathcal{F}_2}}|C_E^{m\divi}-C^{m\divi}|^k +\sum_{\substack{E \in \mathcal{F}^{\text{non-Serre}}}} |C_E^{m\divi}-C^{m\divi}|^k\right).$$
Clearly, the sum over $\mathcal{F}_1$ is zero, so it suffices to bound the remaining two sums.

Let $E/\Q$ be a non-Serre curve. By Remark \ref{upperbound}, Remark \ref{lowerbound}, and \eqref{definitionofaveragedensity}, we have
$$|C_E^{m\divi} - C^{m\divi}|^k < 1.$$
By \cite[Theorem 25]{MR2534114}, there exists an absolute constant $\gamma > 0$ such that
\begin{equation}\label{nonserresum}
    \frac{1}{|\mathcal{F}|}\sum_{\substack{E \in \mathcal{F}^{\text{non-Serre}}}}|C_E^m-C^m|^k \ll \frac{\log^\gamma(\min\{A,B\})}{\sqrt{\min\{A,B\}}}.
\end{equation}
It remains to bound the contribution from $\mathcal{F}_2$. Let $E \in \mathcal{F}_2$. By Theorem
\ref{maintheorem}, one can check that
\begin{equation}\label{localdifference}
    |C_E^{m\divi} - C^{m\divi}|  = C^{m_2\divi} \cdot \prod_{\ell^\alpha \parallel m_1} \frac{\ell^{2\alpha-1}}{|\GL_2(\Z/\ell^\alpha\Z)|} = C^{m_2\divi}  \cdot \frac{1}{m_1^2} \prod_{\ell \mid m_1} \frac{\ell^2}{(\ell-1)^2(\ell+1)}.
\end{equation}
For any odd prime $\ell$, the local factor satisfies
$$\frac{\ell^2}{(\ell-1)^2(\ell+1)} < 1,$$
while it is equal to $4/3$ when $\ell = 2$. Also, it is clear from the definition that $C^{m_2\divi} \leq 1/\varphi(m_2)$. Thus,
\begin{equation}\label{difference}
    |C_E^{m\divi} - C^{m\divi}| \ll \frac{1}{m_1^{2} \varphi(m_2)}.
\end{equation}
For each $d \mid m$, define
$$\mathcal{F}_2(d) \coloneqq \{E \in \mathcal{F}_2 : \gcd(m,m_E^\infty) = d\}.$$
Then,
\begin{equation}\label{averagebound1}
    \sum_{E \in \mathcal{F}_2} |C_E^{m\divi} - C^{m\divi}|^k = \sum_{d\mid m}\sum_{E \in \mathcal{F}_2(d)} |C_E^{m\divi} - C^{m\divi}|^k \ll_k \sum_{d\mid m} \frac{|\mathcal{F}(d)|}{d^{2k} \varphi(m/d)^k}.
\end{equation}
Suppose $E \in \mathcal{F}_2(d)$. Then $\text{rad}(m_E) \mid d$. By Proposition \ref{sizeofmE}, $|\Delta'_E| \mid d$, and in particular, $|\Delta'_E| \leq d$. By \cite[Lemma 22]{MR2534114}, we have
\begin{equation}\label{averagebound2}
    |\mathcal{F}_2(d)| \leq \#\{E \in \mathcal{F}^{\text{Serre}} : |\Delta'_E| \leq d \} \ll B + \log B \cdot A \cdot \log^7 A \cdot d.
\end{equation}
Since $|\mathcal{F}| \asymp 4AB$, it follows from \eqref{averagebound1} and \eqref{averagebound2} that
\begin{equation}\label{F2sum}
    \frac{1}{|\mathcal{F}|} \sum_{E \in \mathcal{F}_2} |C_E^{m\divi} - C^{m\divi}|^k \ll_k \frac{1}{A}\sum_{d\mid m} \frac{1}{d^{2k}\varphi(m/d)^k} + \frac{\log B \cdot \log^7A}{B} \sum_{d\mid m} \frac{1}{d^{2k-1}\varphi(m/d)^k}.
\end{equation}
We define
$$\Sigma_1(m,k) \coloneqq \sum_{d\mid m} \frac{1}{d^{2k}\varphi(m/d)^k} \quad \text{ and } \quad \Sigma_2(m,k) \coloneqq \sum_{d\mid m} \frac{1}{d^{2k-1}\varphi(m/d)^k}.$$
Observe that, for any $d' \mid m$,
$$\frac{d'}{\varphi(d')} = \prod_{\ell \mid d'} \frac{\ell}{\ell-1} \leq \prod_{\ell \mid m} \frac{\ell}{\ell-1} = \frac{m}{\varphi(m)}.$$
Thus, for any $d \mid m$,
$$\frac{1}{\varphi(m/d)^k} = \frac{1}{(m/d)^k} \cdot \frac{(m/d)^k}{\varphi(m/d)^k} \leq \frac{d^k}{m^k} \cdot \frac{m^k}{\varphi(m)^k} = \frac{d^k}{\varphi(m)^k}.$$
Therefore,
\begin{equation}\label{sigma1sigma2}
    \Sigma_1(m,k) \leq \frac{\sigma_{-k}(m)}{\varphi(m)^k} \quad \text{ and } \quad \Sigma_2(m,k) \leq \frac{\sigma_{1-k}(m)}{\varphi(m)^k}.
\end{equation}
Combining \eqref{nonserresum}, \eqref{F2sum}, and \eqref{sigma1sigma2}, we obtain the first part of Theorem \ref{maintheorem2}.

Now we prove the second claim. It suffices to show that $\mathcal{F}_2$ is empty for each value of $m$ in the given list. By Theorem \ref{maintheorem}, if $E \in \mathcal{F}_2$, then $m_E$ must divide $m$. On the other hand, Proposition \ref{sizeofmE} and Lemma \ref{atleast6} imply that $m_E$ is even and $m_E \geq 6$. Thus, if $m$ is odd or $m \leq 4$, then $m_E \nmid m$, and hence $\mathcal{F}_2$ is empty.

It remains to consider $m=8$ and $m=16$. Let $E \in \mathcal{F}_2$. Since $m_E$ divides $m$, Proposition \ref{sizeofmE} and Lemma \ref{atleast6} imply that $m_E=8$. Hence $\Delta'_E=\pm 2$. Since $v_2(m) \in \{3,4\}$, we have $C^{m\divi}_E = C^{m\divi}$ by Theorem \ref{maintheorem}. This contradicts the assumption that $E \in \mathcal{F}_2$. Therefore, $\mathcal{F}_2$ is empty for $m = 8$ and $m =16$ as well. This completes the proof of Theorem \ref{maintheorem2}.

Finally, we consider the case in which $C_E^{m\divi}$ is farthest from $C^{m\divi}$ where $E$ is a Serre curve. For notational convenience, define
$$D(m) \coloneqq \prod_{\ell^\alpha \parallel m} \frac{\ell^{2\alpha-1}}{|\GL_2(\Z/\ell^\alpha \Z)|}.$$

Let $E$ be a Serre curve of adelic level $m_E$ and decompose $m = m_1m_2$ as in Theorem \ref{maintheorem}. Assuming that $C_E^{m\divi} \neq C^{m\divi}$, we have
$$|C_E^{m\divi} - C^{m\divi}| =  D(m_1) C^{m_2\divi}.$$
It is clear from the definition that $C^{m_2\divi}$ attains the largest value of $1$ at $m_2 = 1$. Hence, we restrict our attention to the case that $m_E \mid m \mid m_E^\infty$. By Lemma \ref{atleast6}, this forces that $m$ must be at least $6$. We first determine the upper bound of $D(m)$.
\begin{lemma}\label{largesterror}
    For any even $m \geq 6$, $D(m) \leq 1/48$.  
\end{lemma}
\begin{proof}
 From \eqref{localdifference}, one can easily observe that $D(p^\alpha) \leq D(p)$ for any prime power $p^\alpha$, and that $D(q) < D(p)$ for any two distinct primes $p < q$.

 Write $m = 2^\alpha m'$, where $m'$ is odd. Suppose first that $\alpha \leq 2$. Since $m \geq 6$, we have $m' \geq 3$. Hence,
    $$D(2^\alpha m') =D(2^\alpha) D(m') \leq D(2)D(3) = D(6) = \frac{1}{48}.$$
    Now suppose that $\alpha \geq 3$. Then
   $$D(2^\alpha m') = D(2^\alpha) D(m') \leq D(2^3) = \frac{1}{48}.$$
    This completes the proof.
\end{proof}
For a Serre curve $E$, the proof of Lemma \ref{largesterror} shows that the largest possible value of $|C_E^{m\divi}-C^{m\divi}|$ can occur only when $m=6$ or $m=8$. However, by Theorem \ref{maintheorem}, we have $C_E^{8\divi} = C^{8\divi}$. Thus, the largest difference occurs only when $m = 6$.

Let $E$ be a Serre curve with $\Delta'_{E} = -3$. By Proposition \ref{sizeofmE}, its adelic level is 6. By Theorem \ref{maintheorem}, we have
$$C_{E}^{6\divi} = C^{6\divi} + \prod_{\ell^\alpha \parallel 6} \frac{-\ell^{2\alpha-1}}{|\GL_2(\Z/\ell^\alpha\Z)|} = \frac{7}{24} + \frac{1}{48} = \frac{5}{16},$$
which almost doubles the naive expectation of $1/6$.

\begin{example}
    Consider $E:Y^2=X^3-48X+272$. Its LMFDB label is \texttt{135.a1}. This is a Serre curve of adelic level $6$. One can check that
    $$C^{6\divi} = \frac{7}{24} = 0.291667\ldots \qquad \text{and} \qquad C^{6\divi}_E = \frac{5}{16} = 0.3125.$$
    Thus, $C_E^{6\divi}$ is about $7.14\%$ larger than the average density.

    Our Magma computation found that
    $$\frac{\pi_E^{6\divi}(10^8)}{\#\{p\leq 10^8 : p\nmid N_E\}} = \frac{1800451}{5761452} = 0.312500 \ldots,$$
    which is consistent with $C_E^{6\divi}$.
\end{example}

Now, we check the case where $C_E^{m\divi}$ is furthest below $C^{m\divi}$. By Theorem \ref{maintheorem}, this can occur when $m_2 = 1$ and $m_1$ has an odd number of distinct prime factors, one of which must be $2$. Moreover, Lemma \ref{largesterror} shows that the discrepancy is largest when $m_1$ is as small as possible. Thus, the only remaining candidates are $m_1 = 2^5$, since the cases $m_1 = 2^3$ and $m_1=2^4$ have already been ruled out, and $m_1= 2\cdot 3 \cdot 5$.

\begin{example}
    Consider $E:Y^2=X^3-183X-993$. Its LMFDB label is \texttt{1080.a1}. This is a Serre curve of adelic level $30$. One can check that
    $$C^{30\divi} = \frac{161}{2304} = 0.069878\ldots \qquad \text{and} \qquad C^{30\divi}_E = \frac{107}{1536} = 0.069661\ldots.$$
    Thus, $C_E^{30\divi}$ is about $0.31\%$ smaller than the average density.

    Our Magma computation found that
    $$\frac{\pi_E^{30\divi}(10^8)}{\#\{p\leq 10^8 : p\nmid N_E\}} = \frac{401612}{5761452} = 0.069707 \ldots,$$
    which is consistent with the theoretical density $C_E^{30\divi}$.
\end{example}

\begin{example}
    Consider $E:Y^2=X^3+125X-1250$. Its LMFDB label is \texttt{200.a1}. This is a Serre curve of adelic level $8$. Taking $m=32$, one can check that
    $$C^{32\divi} = \frac{47}{768} = 0.061198\ldots \qquad \text{and} \qquad C^{32\divi}_E = \frac{23}{384} = 0.059896\ldots.$$
    Thus, $C_E^{32\divi}$ is about $2.13\%$ smaller than the average density.

    Our Magma computation found that
    $$\frac{\pi_E^{32\divi}(10^8)}{\#\{p\leq 10^8 : p\nmid N_E\}} = \frac{344797}{5761453} = 0.059845 \ldots,$$
    which is consistent with the theoretical density $C_E^{32\divi}$.
\end{example}
As one can see, the discrepancy is larger in the case $m_E=8$ and $m=32$.

\bibliographystyle{amsplain}
\bibliography{References}

\end{document}